\documentclass{elsarticle}

\usepackage[all,cmtip]{xy}
\usepackage{amsmath, amssymb}
\usepackage{graphics}
\usepackage{caption}
\usepackage{subcaption}
\usepackage{amsthm}
\usepackage{algorithm}
\usepackage{algorithmic}
\usepackage{booktabs}
\usepackage{color}

\newtheorem{theorem}{{\bf Theorem}}
\newtheorem{remark}{{\bf Remark}}
\newtheorem{definition}[theorem]{{\bf Definition}}
\newtheorem{corollary}[theorem]{{\bf Corollary}}
\newtheorem{proposition}[theorem]{{\bf Proposition}}
\newtheorem{lemma}[theorem]{{\bf Lemma}}
\newtheorem{example}{{\bf Example}}

\def\bfa{\bf a}

\def\bfx{\boldsymbol{x}}

\def\bfp{\boldsymbol{p}}
\def\bfq{\boldsymbol{q}}
\def\bfv{\bf v}

\def\CC{\mathbb{C}}
\def\RR{\mathbb{R}}

\def\CCC{\mathcal{C}}

\begin{document}

\begin{frontmatter}
\title{Recognizing projections of rational curves.}

\author[a]{Juan Gerardo Alc\'azar\fnref{proy}}
\ead{juange.alcazar@uah.es}
\author[a]{Carlos Hermoso}
\ead{carlos.hermoso@uah.es}

\address[a]{Departamento de F\'{\i}sica y Matem\'aticas, Universidad de Alcal\'a,
E-28871 Madrid, Spain}

\fntext[proy]{
Supported by the Spanish Ministerio de Econom\'{\i}a y Competitividad and by the European Regional Development Fund (ERDF), under the project  MTM2014-54141-P. Member of the Research Group {\sc asynacs} (Ref. {\sc ccee2011/r34}) }

\begin{abstract}
Given two rational, properly parametrized space curves ${\mathcal C}_1$ and ${\mathcal C}_2$, where $\CCC_2$ is contained in some plane $\Pi$, we provide an algorithm to check whether or not there exist perspective or parallel projections mapping $\CCC_1$ onto $\CCC_2$, i.e. to recognize $\CCC_2$ as the projection of $\CCC_1$. In the affirmative case, the algorithm provides the eye point(s) of the perspective transformation(s), or the direction(s) of the parallel projection(s). The problem is mainly discussed from a symbolic point of view, but an approximate algorithm is also included. 
\end{abstract}

\end{frontmatter}

\section{Introduction}\label{section-introduction}

In this paper, we address the following geometric problem: \emph{given} a rational space curve $\CCC_2$, lying on a plane $\Pi$, and another rational space curve $\CCC_1$, not necessarily planar, \emph{check} whether or not there exist perspective or parallel projections mapping $\CCC_1$ onto $\CCC_2$, and \emph{find} them in the affirmative case. 

Our problem can be translated into the context of Computer Vision. For this purpose, we recall \cite{HZ} the simplest camera model, known as the \emph{pinhole camera}. In this model, a camera is modeled as a pair $(\bfa,\Pi)$, where $\bfa$ is called the \emph{eye point} of the camera, and ${\bf \Pi}$ is the \emph{image plane}: then, given an object $\Omega\subset {\Bbb R}^3$, the photograph of $\Omega$ taken by the camera is the projection of $\Omega$ from the point $\bfa$ onto the plane ${\bf \Pi}$ (see Fig. \ref{fig0}). The eye point can be allowed to be at infinity, in which case we have a parallel projection from a certain direction. 

\begin{figure}
$$\begin{array}{c}
\includegraphics[scale=0.3]{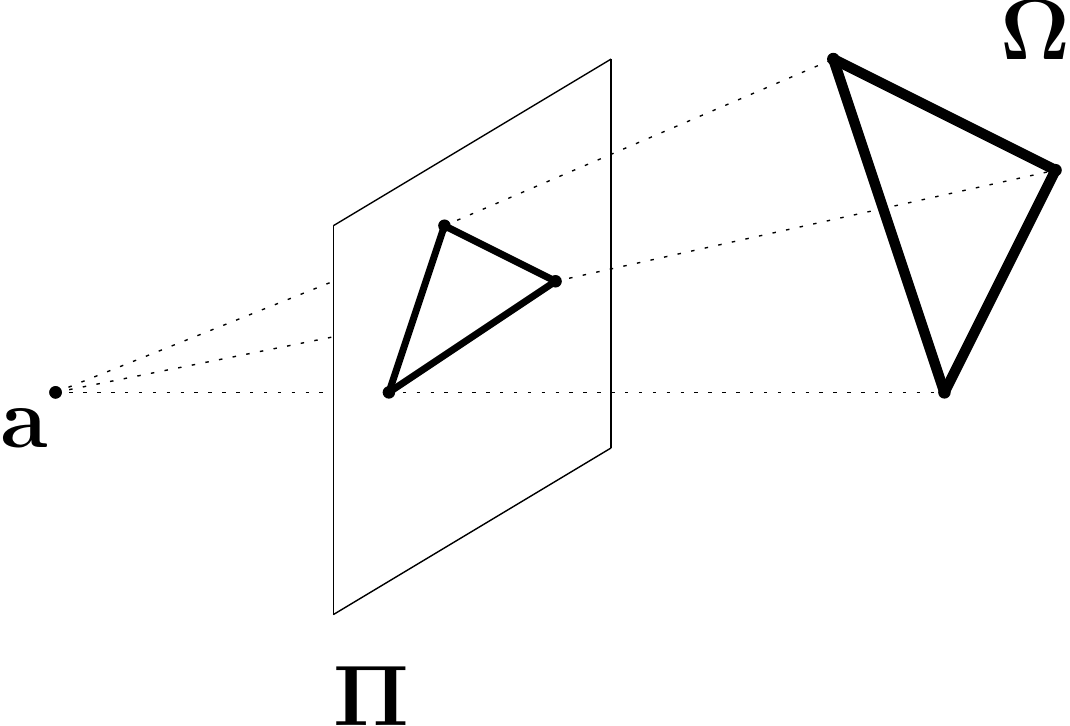}
\end{array}$$
\caption{Pinhole camera model.}\label{fig0}
\end{figure}

Therefore, in this context our problem can be translated as whether or not $\CCC_2$ can be regarded as a {\it photograph} of $\CCC_1$, taken with a camera where the image plane is known (it is the plane $\Pi$ containing $\CCC_2$), but where the eye point is unknown. 

 A more general problem is treated in \cite{Burdis}. In \cite{Burdis} the input is a pair of algebraic curves, ${\mathcal D}_1\subset {\Bbb R}^3$ and ${\mathcal D}_2\subset {\Bbb R}^2$, not necessarily rational,and the question is to check if there exists some camera where ${\mathcal D}_2$ is the photograph of ${\mathcal D}_1$: in other words, to find the positions of the eye point {\it and} the image plane, if any, such that ${\mathcal D}_2$ is the photograph of ${\mathcal D}_1$ taken from the camera eye point. This problem is known as the \emph{object-image correspondence} problem, and amounts to recognizing images without any clue on the parameters of the camera used to take the photograph. 

Since in our case we assume that the image plane is known, the problem here can be considered as a weak version of the problem in \cite{Burdis}. Because the problem is simpler, we can find a solution computationally simpler than that of \cite{Burdis}, too. In \cite{Burdis} the problem is solved by deciding whether the curve ${\mathcal D}_2$ is equivalent to some curve in a family of planar curves, computed from ${\mathcal D}_1$, under an action of the projective or the affine group. In turn, this is done by using differential invariants. Computationally, the question boils down to solving a quantifier elimination problem with five variables, in the case of perspective projections, and with four variables in the case of parallel projections. Since the differential invariants used in \cite{Burdis} are high-order (5 and 6 for affine actions, 7 and 8 for projective actions), the elimination problem can be hard.

In our case, we restrict to rational curves parametrized over ${\Bbb Q}$, and use a very different approach. We observe that any projection between $\CCC_1$ and $\CCC_2$ corresponds to a rational function $\psi$ between the parameter spaces; furthermore, $\psi$ is shown to be a \emph{M\"obius transformation} in the case of \emph{non-degenerate} projections between $\CCC_1$ and $\CCC_2$, i.e. projections which are injective for almost all points of $\CCC_2$. We show that the rational functions $\psi$ potentially corresponding to projections from eye points with rational coordinates can be efficiently computed by means of standard bivariate factoring techniques over the rationals. In order to also find the projections from non-rational eye points we need bivariate factoring over the reals, i.e. an absolute factorization. Furthermore, once $\psi$ is computed, checking if it gives rise to some projection between $\CCC_1$ and $\CCC_2$ is easy. 

As an alternative to computing an absolute factorization, we also provide another algorithm, both in symbolic and approximate versions, which takes advantage of the existence of $\psi$ without actually computing it. In general, the symbolic version of this last algorithm requires to compute the primitive element of an algebraic extension ${\Bbb Q}(\alpha,\beta)$, which can be costly. However, the approximate version of the algorithm, where algebraic numbers are numerically approximated, is fast. In this last case, we do not check if $\CCC_2$ is the projection of $\CCC_1$, but if $\CCC_2$ is ``approximately" the projection of $\CCC_1$. Furthermore, this approximate algorithm is well-suited for curves whose defining parametrizations are known only up to a certain precision, i.e. with floating point coefficients, which is closer to applications.  


\vspace{0.5 cm}

\noindent
{\bf Acknowledgements.} We thank Ron Goldman for some fruitful discussions on the problem. We also thank Sonia Rueda for her help with the estimation of Hausdorff distances, in Example \ref{exfin}.

\section{Projective and parallel projections.} \label{sec-prelim}

\noindent Throughout the paper, we consider two rational space curves $\CCC_1, \CCC_2 \subset \RR^3$, $\CCC_1\neq \CCC_2$. Such curves are algebraic, irreducible and can be parametrized by rational maps
\begin{equation}\label{eq:parametrizations}
{\bfx}_j: \RR \dashrightarrow \CCC_j\subset \RR^3, \qquad \bfx_j(t) = \big( x_j(t), y_j(t), z_j(t) \big),\qquad j = 1,2.
\end{equation}
We will suppose that $\bfx_1,\bfx_2$ have coefficients in ${\Bbb Q}$, and that $\CCC_2$ is contained in some plane $\Pi$; however, $\CCC_1$ is not necessarily planar. Also, we will exclude the case when $\CCC_1, \CCC_2$ are two planar curves contained in the same plane. The components $x_j, y_j, z_j$ of $\bfx_j$ are real, rational functions of $t$, therefore defined for all but a finite number of values of $t$. Nevertheless, at certain moments we will consider $x_j,y_j,z_j$ as functions from $\CC$ to $\CC$. 

For $j=1,2$, the parametrization $\bfx_j$ generates all the points of $\CCC_j$ except perhaps the point $P_j^{\infty}=\mbox{lim}_{t\to \infty}\bfx_j(t)$, which is affine whenever the degrees of the numerators of $x_j,y_j,z_j$ are less or equal than the degrees of the denominators. We will assume that the parametrizations in \eqref{eq:parametrizations} are \emph{proper}, i.e., birational or, equivalently, injective except for perhaps finitely many values of $t$. This can be assumed without loss of generality, since any rational curve can be properly reparametrized. For these claims and other results on properness, the interested reader can consult \cite{SWPD} for plane curves and \cite[\S 3.1]{A12} for space curves. 

\vspace{3 mm}

Now let us introduce projective and parallel projections. Let $\Pi$ be a plane. The \emph{parallel projection onto $\Pi$, in the direction of a nonzero vector $\bfv \in {\Bbb R}^3$}, is the transformation in 3-space that maps every (complex or real) point $\bfp$ onto the intersection $\bfq$ of $\Pi$ and the line through $\bfp$ parallel to $\bfv$. The \emph{perspective projection onto $\Pi$ from a point $\bfa$}, called the {\it eye point}, is the transformation in 3-space that maps every point $\bfp$ onto the intersection $\bfq$ of $\Pi$ and the line connecting $\bfa$ and $\bfp$. These definitions are illustrated by Fig \ref{fig1}. 

\begin{figure}
$$\begin{array}{c}
\includegraphics[scale=0.3]{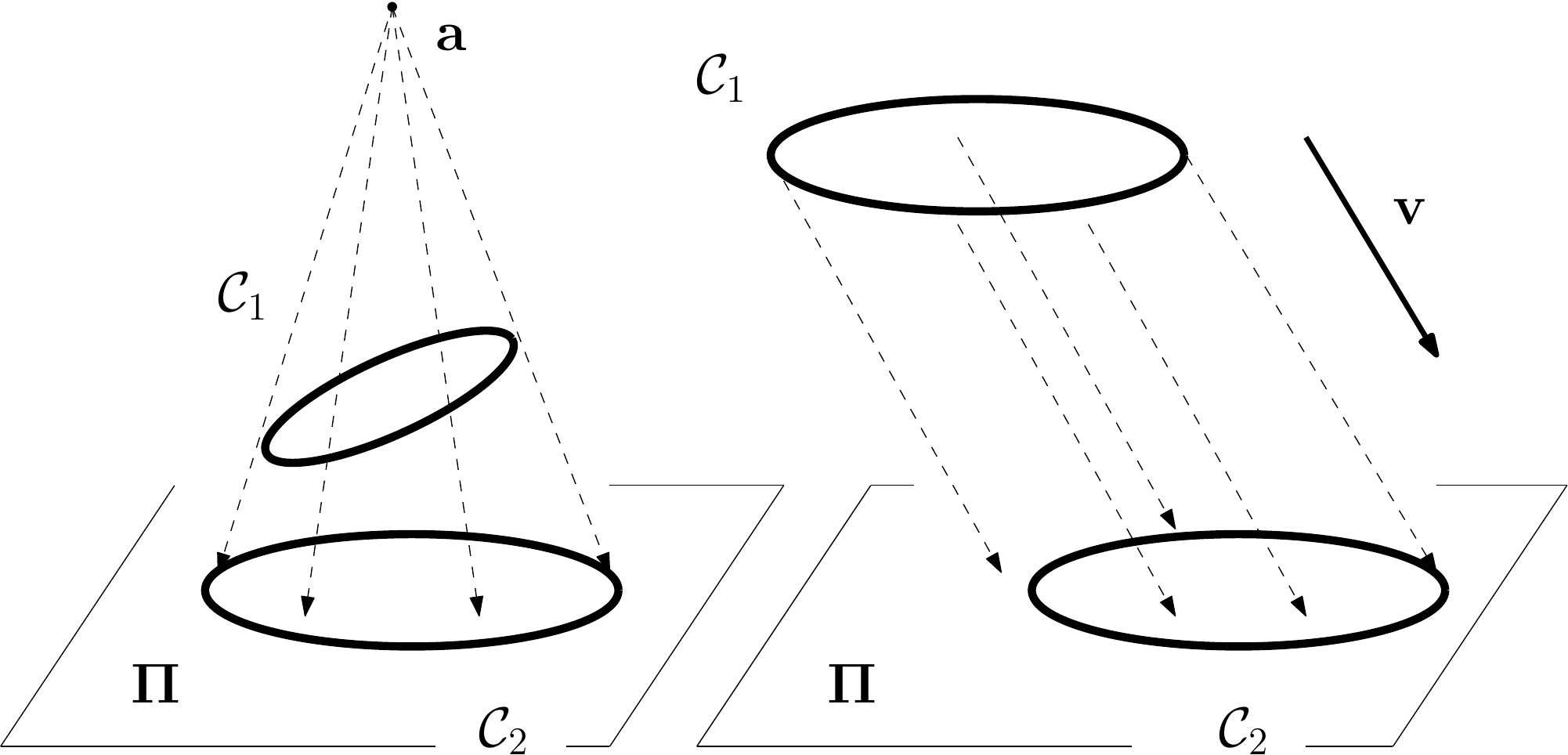}
\end{array}$$
\caption{Perspective projection (left), parallel projection (right).}\label{fig1}
\end{figure}

Parallel and perspective projections can be unified when we move to a projective setting. The (complex) {\it projective space}, ${\Bbb P}^3_{\Bbb C}$, is the set of 4-tuples $\tilde{\bfp}=[x:y:z:\omega]$, where $x,y,z,\omega\in {\Bbb C}$ and at least one them is nonzero, such that two such 4-tuples are considered equal when they are proportional to each other. If $\omega=0$ we have a {\it point at infinity}, while $\omega\neq 0$ corresponds to an {\it affine point}. In particular, if $\tilde{\bfp}$ is affine then it can be represented by the 4-tuple
\[\left[\frac{x}{\omega}:\frac{y}{\omega}:\frac{z}{\omega}:1\right].\]The \emph{projective closure} of ${\mathcal C}_j$ is the curve in ${\Bbb P}^3_{\Bbb C}$ whose affine part is ${\mathcal C}_j$; the projective closure can be parametrized as 
\[\tilde{\bfx}(t)=\left[\tilde{x}_j(t):\tilde{y}_j(t):\tilde{z}_j(t):\tilde{\omega}_j(t)\right],\]where $x_j(t)=\frac{\tilde{x}_j(t)}{\tilde{\omega}_j(t)}$, $y_j(t)=\frac{\tilde{y}_j(t)}{\tilde{\omega}_j(t)}$, $z_j(t)=\frac{\tilde{z}_j(t)}{\tilde{\omega}_j(t)}$. 
For simplicity, we will use the same notation for a curve and its projective closure; it will be clear from the context whether we are working with one or the other. 

In this context, parallel or perspective projections are treated in the same way: in the case of perspective projections the eye point is affine, and in the case of parallel projections, the eye point is at infinity. So both projections \cite[\S 13]{G} can be represented by a projective transformation
\[
\begin{array}{cc}
\begin{bmatrix} 
x'\\
y'\\
z'\\
\omega'
\end{bmatrix}= & 
\underbrace{\begin{bmatrix} 
p_{11} & p_{12} & p_{13} & p_{14} \\
p_{21} & p_{22} & p_{23} & p_{24} \\
p_{31} & p_{32} & p_{33} & p_{34} \\
p_{41} & p_{42} & p_{43} & p_{44}
\end{bmatrix}}_P\cdot 
\begin{bmatrix} 
x\\
y\\
z\\
\omega
\end{bmatrix}.
\end{array}
\]
If $\tilde{\bfa}:=[\tilde{a}_1:\tilde{a}_2:\tilde{a}_3:\tilde{a}_4]$ denotes the eye point of the projection, and the implicit equation of the projection plane $\Pi$ is $Ax+By+Cz+D=0$, an easy computation shows that the matrix $P$ is 

\begin{equation}\label{matP}
\hspace{-1.5 cm}P=\begin{bmatrix} 
-\tilde{a}_2B-\tilde{a}_3C-\tilde{a}_4D & \tilde{a}_1B & \tilde{a}_1C & \tilde{a}_1D\\
\tilde{a}_2A & -\tilde{a}_1 A-\tilde{a}_3C-\tilde{a}_4D & \tilde{a}_2C & \tilde{a}_2D\\
\tilde{a}_3A & \tilde{a}_3B & -\tilde{a}_1A-\tilde{a}_2B-\tilde{a}_4D & \tilde{a}_3 D\\
\tilde{a}_4 A & \tilde{a}_4 B & \tilde{a}_4 C & -\tilde{a}_1A-\tilde{a}_2B-\tilde{a}_3C
\end{bmatrix}.
\hspace{10 cm}
\end{equation}
In particular, not every $4\times 4$ matrix represents a projection. Any matrix $P$ representing a projection satisfies that $\mbox{rank}(P)=3$; therefore the dimension of $\mbox{Ker}(P)$ is 1, and any basis of $\mbox{Ker}(P)$ provides projective coordinates for the eye point of the transformation. Furthermore, the affine part of the image $\mbox{Im}(P)$ defines the projection plane $\Pi$. If we denote the matrix consisting of the first three columns of $P$ by $M$, perspective projections satisfy that $\mbox{rank}(M)=3$, while parallel projections satisfy that $\mbox{rank}(M)<3$. We will represent the projection onto a plane $\Pi$ from an eye point $\tilde{\bfa}$ as ${\mathcal P}_{\tilde{\bfa}}$, so as to make the eye point explicit. Furthermore, when moving to the affine space, parallel projections correspond to \emph{affine} 3-D transformations, while perspective projections correspond to \emph{rational} 3-D transformations, where the numerator and denominator of each component has degree one. 

In the paper we will work with \emph{real} projections, so that the eye point, and therefore the matrix $P$, are real. Additionally, given a projection ${\mathcal P}_{\tilde{\bfa}}$ and a curve ${\mathcal C}$, the \emph{projection of} ${\mathcal C}$, ${\mathcal P}_{\tilde{\bfa}}({\mathcal C})$, is the image of ${\mathcal C}$ under ${\mathcal P}_{\tilde{\bfa}}$.

\begin{proposition}\label{curvesproj}
Let $P$ be a $4\times 4$ matrix defining a projection ${\mathcal P}_{\tilde{\bfa}}$, and let ${\mathcal C}$ be a projective rational curve properly parametrized by $\tilde{\bfx}(t)$, with at least one affine point. If ${\mathcal C}$ is not a line going through $\tilde{\bfa}$, then ${\mathcal P}_{\tilde{\bfa}}({\mathcal C})$ is a rational curve. 
\end{proposition}

\begin{proof}
Let $\tilde{\bfx}(t)=[\tilde{x}(t):\tilde{y}(t):\tilde{z}(t):\tilde{\omega}(t)]$. If ${\mathcal P}_{\tilde{\bfa}}({\mathcal C})$ is a curve then $P\cdot \tilde{\bfx}(t)$ parametrizes ${\mathcal P}_{\tilde{\bfa}}({\mathcal C})$, so ${\mathcal P}_{\tilde{\bfa}}({\mathcal C})$ is clearly rational. So suppose that ${\mathcal P}_{\tilde{\bfa}}({\mathcal C})$ is not a curve. Then there exists a polynomial $\lambda(t)$, and $\alpha_i\in {\Bbb C}$, $i=1,\ldots,4$, such that 
\[P\cdot \begin{bmatrix}\tilde{x}(t)\\ \tilde{y}(t)\\ \tilde{z}(t)\\ \tilde{\omega}(t)\end{bmatrix}=\lambda(t)\cdot \begin{bmatrix} \alpha_1\\ \alpha_2 \\ \alpha_3 \\ \alpha_4\end{bmatrix}.\]
In this situation, for $i=1,\ldots,4$ we have that 
\begin{equation}\label{equ}
p_{i1}\cdot \tilde{x}(t)+p_{i2}\cdot \tilde{y}(t)+p_{i3}\cdot \tilde{z}(t)+p_{i4}\cdot \tilde{\omega}(t)=\lambda(t)\cdot \alpha_i.
\end{equation}
Therefore, there exist $\mu_1,\mu_2,\mu_3\in{\Bbb C}$ such that 
\begin{equation}\label{lasec}
\hspace{-1 cm}\begin{array}{ccc}
(p_{11}-\mu_1 \cdot p_{21})\tilde{x}(t)+(p_{12}-\mu_1 \cdot p_{22})\tilde{y}(t)+(p_{13}-\mu_1 \cdot p_{23})\tilde{z}(t)+(p_{14}-\mu_1 \cdot p_{24})\tilde{\omega}(t) & = & 0\\
(p_{11}-\mu_2 \cdot p_{31})\tilde{x}(t)+(p_{12}-\mu_2 \cdot p_{32})\tilde{y}(t)+(p_{13}-\mu_2 \cdot p_{33})\tilde{z}(t)+(p_{14}-\mu_2 \cdot p_{34})\tilde{\omega}(t) & = & 0\\
(p_{11}-\mu_3 \cdot p_{41})\tilde{x}(t)+(p_{12}-\mu_3 \cdot p_{42})\tilde{y}(t)+(p_{13}-\mu_3 \cdot p_{43})\tilde{z}(t)+(p_{14}-\mu_3 \cdot p_{44})\tilde{\omega}(t) & = & 0.
\end{array}
\end{equation}
Notice that each of the above equations expresses that ${\mathcal C}$ is contained in a certain plane $\Pi_k$, $k=1,2,3$. Since $\mbox{rank}(P)=3$, at least two of these planes are different, so ${\mathcal C}$ is contained in the intersection of two different planes, and therefore ${\mathcal C}$ is a straight line. Furthermore, since by hypothesis ${\mathcal C}$ has at least one affine point, ${\mathcal C}$ must be an affine straight line. Hence, let $\tilde{\bf x}_0+t\cdot {\bf v}$, where $\tilde{\bf x}_0=[x_0:y_0:z_0:1]$, ${\bf v}=[v_1:v_2:v_3:0]$ be a parametrization of ${\mathcal C}$. Then $P\cdot \tilde{\bfx}(t)=P\cdot \tilde{\bf x}_0+tP\cdot {\bf v}$. Now recall that $P\cdot \tilde{\bfx}(t)=\lambda(t)\cdot {\bf \alpha}$, where ${\bf \alpha}=[\alpha_1:\alpha_2:\alpha_3:\alpha_4]$. If $\lambda(t)$ is a constant, then $P\cdot {\bf v}=0$, so ${\bf v}$ is the eye point of the projection defined by $P$. Therefore, ${\mathcal C}$ goes through the projective point defined by ${\bf v}$. However, by hypothesis this cannot happen. So $\lambda(t)$ cannot be a constant, in which case $P\cdot \tilde{\bf x}_0$ and $P\cdot {\bf v}$ are proportional, because each one must be proportional to $\alpha$. Therefore there exist $a,b\in {\Bbb C}$ such that $a{\bf x}_0+b{\bf v}\in \mbox{Ker}(P)$, which implies that ${\mathcal C}$ contains the eye point of the projection. Again, by hypothesis this cannot happen. We deduce that ${\mathcal P}_{\tilde{\bfa}}({\mathcal C})$ must be a curve, so the proposition is proven. 
\end{proof}

\section{Statement of the problem.}\label{sec-state}

In order to formally state the problem we want to solve, we first need the following definition.

\begin{definition}\label{defpro} Let $\tilde{\bfa}\in {\Bbb P}^3_{\Bbb R}$, let $\CCC_1$, $\CCC_2$ be two space rational curves, where $\CCC_2$ is contained in a plane $\Pi$, and let ${\mathcal P}_{\tilde{\bfa}}$ be the projection from the eye point $\tilde{\bfa}$ onto the plane $\Pi$. We say that $\CCC_2$ is the projection of $\CCC_1$ from $\tilde{\bfa}$, i.e. $\CCC_2={\mathcal P}_{\tilde{\bfa}}(\CCC_1)$, if every (projective) point of $\CCC_2$ is the projection of some (projective) point of $\CCC_1$. Furthermore, we will say that the projection is \emph{non-degenerate} if ${\mathcal P}_{\tilde{\bfa}}\vert_{{\mathcal C}_1}^{-1}$ is injective for almost all points of $\CCC_2$; otherwise we will say that the projection is \emph{degenerate}.
\end{definition} 

Observe that Definition \ref{defpro} is consistent. Indeed, because of Proposition \ref{curvesproj} and since $\CCC_1$, $\CCC_2$ are both irreducible, the projection ${\mathcal P}_{\tilde{\bfa}}(\CCC_1)$ of $\CCC_1$ either has finitely many points in common with $\CCC_2$ (in which case certainly $\CCC_2$ is not the projection of $\CCC_1$ from $\tilde{\bfa}$), or completely coincides with $\CCC_2$. However, this requires working over the complex projective space: for instance, according to Definition \ref{defpro} the parabola $\{y=x^2,z=0\}$ is the projection of the space curve parametrized by $(t,t^2,1/t)$ from the point $[0:0:1:0]$, but the origin is the image of a point at infinity. Also, $\{y=x^2,z=0\}$ is the (degenerate) projection of the space curve $(t^2,t^4,t)$ from the point $[0:0:1:0]$; however, all the affine points $(a,a^2,0)$ of the parabola with $a<0$ come from complex, affine points $\left(-|a|,|a|^2,i\sqrt{|a|}\right)$ of the space curve. 

On the other hand, the notion of degeneracy arising in Definition \ref{defpro} implies the existence of either two different (possibly complex) branches of $\CCC_1$ projecting onto a same branch of $\CCC_2$, or some branch of $\CCC_1$ collapsing onto a point of $\CCC_2$ under projection. The notion of degeneracy is, in general, also defined over the complex numbers. It can even happen that degeneracy occurs only over the complex, but not the real numbers. For instance, let $\CCC_1$ be the curve parametrized by $\bfx_1(t)=(t^3,t^6,t)$, and let $\CCC_2$ be the parabola parametrized by $\bfx_2(s)=(s,s^2,0)$. It is clear that for every real point $\bfp=(a,a^2,0)\in \CCC_2$, with $a\in {\Bbb R}$, there is just one real point of $\CCC_1$ projecting onto $\bfp$, namely the point $(a,a^2,\sqrt[3]{a})$. However, since $\sqrt[3]{a}$ has three different complex values for $a\neq 0$, there are three complex points of $\CCC_1$ projecting onto $\bfp$, namely the points $(a,a^2,\sqrt[3]{a}\cdot \xi)$, where $\xi^3=1$. So every branch of $\CCC_2$ comes from three different complex branches of $\CCC_1$.

Now we can state the problem we want to solve. Although in the paper we will mostly work in the affine space, the language of projective space is useful to state the problem in a clearer and more general way.

\begin{itemize}
\item {\bf Projection Problem:} given two affine, space real algebraic curves $\CCC_1$, $\CCC_2$, where $\CCC_2$ is contained in a plane $\Pi$, properly parametrized by $\bfx_1(t)$, $\bfx_2(s)$, {\bf check} if there exists $\tilde{\bfa}\in {\Bbb P}^3_{\Bbb R}$ (and find it in the affirmative case) such that $\CCC_2$ is the projection of $\CCC_1$ from $\tilde{\bfa}$.
\end{itemize}

Observe that once $\tilde{\bfa}$ is computed we can find out the nature of the projection: if $\tilde{\bfa}$ is a point at infinity we have a parallel projection, and if $\tilde{\bfa}$ is affine we have a perspective projection. Additionally, if either $\CCC_1$ or $\CCC_2$ is a straight line, then $\CCC_2={\mathcal P}_{\tilde{\bfa}}(\CCC_1)$ implies that both  $\CCC_1$ and $\CCC_2$ are contained in a same plane, which is a case we excluded at the beginning of the section; so in the rest of the paper, we will assume that $\CCC_1$, $\CCC_2$ are not straight lines.

The parallel or perspective projection mapping $\CCC_1$ onto $\CCC_2$ is not necessarily unique. For instance, let $\CCC_1$, $\CCC_2$ be the two irreducible components of the space curve obtained by intersecting the cylinders $x^2+y^2=1$, $x^2+z^2=1$ (see Figure \ref{fig:cylinders}).  One can check that the projections parallel to the $y$-axis and the $z$-axis both transform $\CCC_1$ into $\CCC_2$. Similarly, let $\CCC_1$ be $\{x^2+y^2=1,z=1\}$ and let $\CCC_2$ be $\{x^2+y^2=2,z=0\}$ (see Fig. \ref{fig:unique}). These curves are two circles of radii 1 and $\sqrt{2}$, located in the planes $z=1$ and $z=0$, with centers on the $z$-axis. In this case there are two different perspective projections transforming $\CCC_1$ into $\CCC_2$, one from the point $\left(0,0,\frac{\sqrt{2}}{\sqrt{2}-1}\right)$, and another one from the point $\left(0,0,\frac{\sqrt{2}}{\sqrt{2}+1}\right)$. This example also shows that two rational curves $\CCC_1$ and $\CCC_2$ parametrized over ${\Bbb Q}$ can however be related by a projection from a point with non-rational coordinates. 

\begin{figure}
\begin{center}
\includegraphics[scale=0.4]{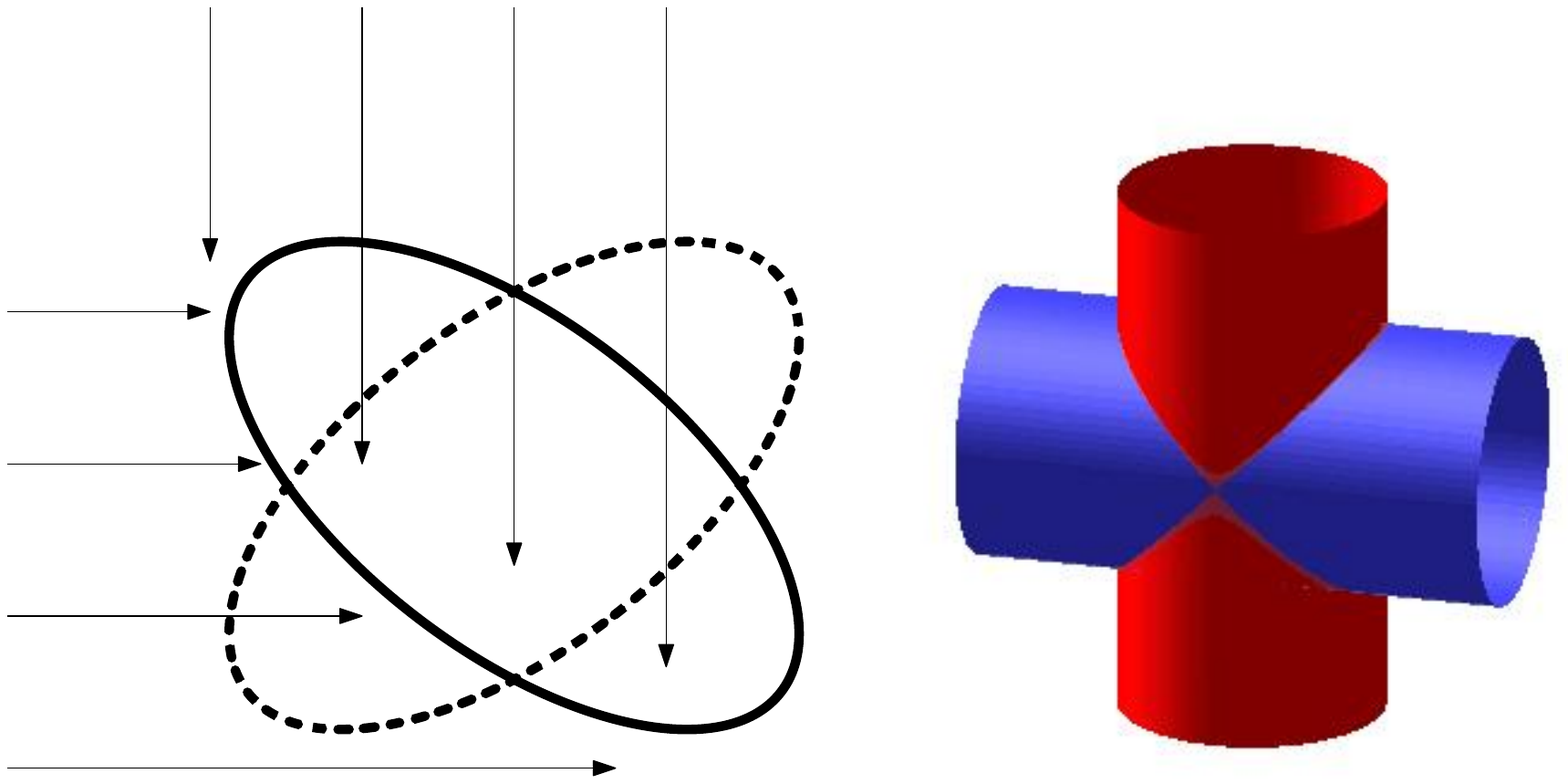} 
\end{center}
\caption{Right: intersection of the cylinders $x^2+y^2=1$, $x^2+z^2=1$. Left: the two irreducible curves we get when intersecting $x^2+y^2=1$, $x^2+z^2=1$ can be projected onto each other in two different directions. }\label{fig:cylinders}
\end{figure}

\begin{figure}
\begin{center}
\includegraphics[scale=0.4]{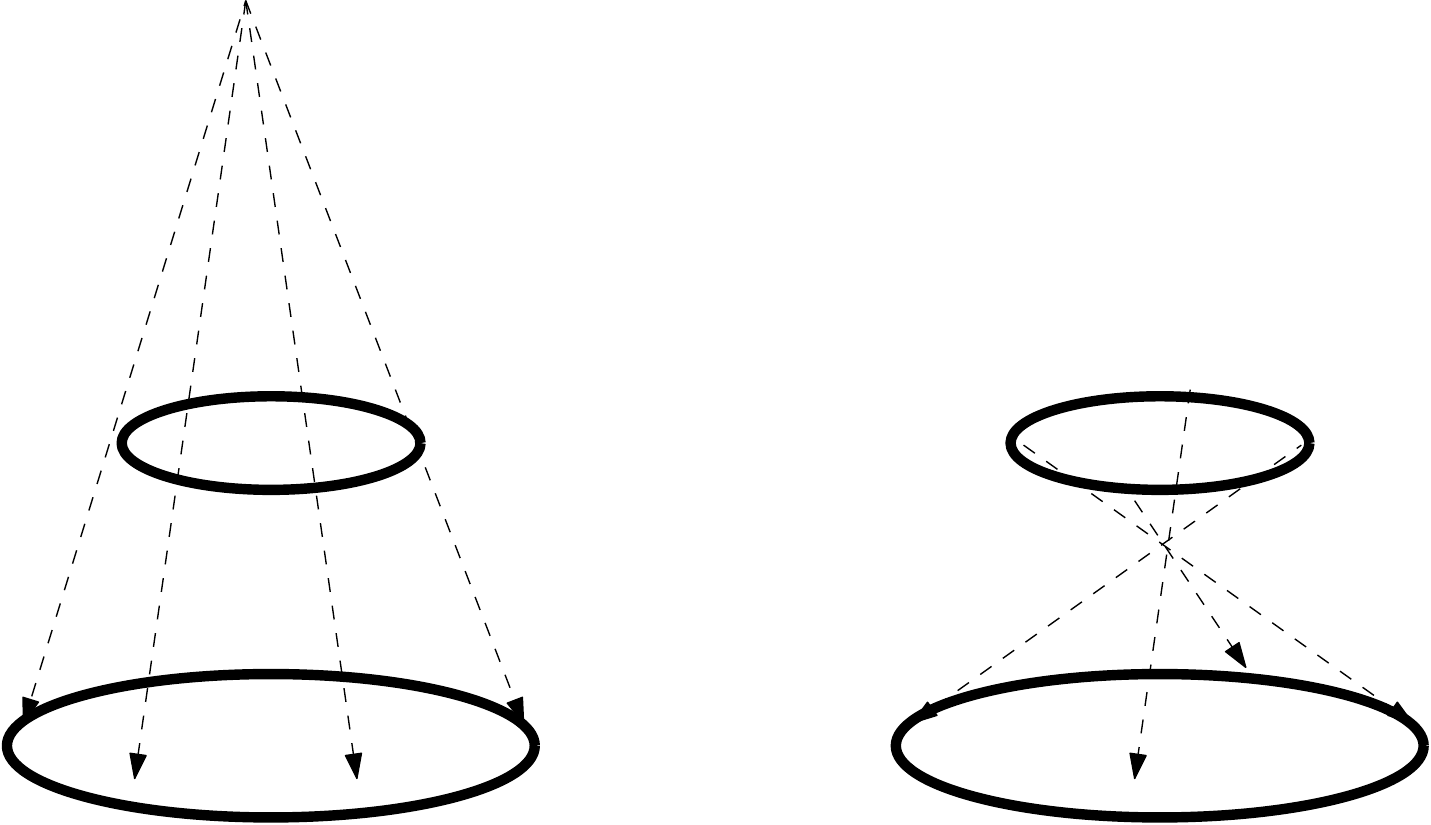} 
\end{center}
\caption{The circles $\{x^2+y^2=1,z=1\}$ and  $\{x^2+y^2=2,z=0\}$ are related by two different perspective projections. }\label{fig:unique}
\end{figure}

Notice that there can be a parallel projection and a perspective projection simultaneously transforming $\CCC_1$ into $\CCC_2$, too. For instance, let $\CCC_1$ be $\{x^2+y^2=1,z=1\}$ and let $\CCC_2$ be $\{x^2+y^2=1,z=0\}$. It is clear that $\CCC_2$ is the projection of $\CCC_1$ parallel to the direction of the $z$-axis. But $\CCC_2$ is also the perspective projection of $\CCC_1$ from the point $(0,0,1/2)$.

\section{A factoring-based strategy.}\label{sec-general}

Suppose that $\CCC_2={\mathcal P}_{\tilde{\bfa}}({\mathcal C}_1)$, where $\tilde{\bfa}=[\tilde{a}_1:\tilde{a}_2:\tilde{a}_3:\tilde{a}_4]$, $\tilde{a}_4\neq 0$. Hence, for almost every affine point $\bfx_2(s)\in \CCC_2$ there exists another affine point $\bfx_1(t)\in \CCC_1$, and a scalar $\widehat{\lambda}=\widehat{\lambda}(t,s)$, such that the vector connecting the points $\bfx_1(t)$ and $\bfx_2(s)$ is parallel to the vector connecting the point $\bfx_1(t)$ and the eye point \[{\bf a}=\left(\frac{\tilde{a}_1}{\tilde{a}_4},\frac{\tilde{a}_2}{\tilde{a}_4},\frac{\tilde{a}_3}{\tilde{a}_4}\right),\]with $\widehat{\lambda}$ being the proportionality constant between these two vectors. So for almost all $s$ there exist $t,\widehat{\lambda}$ such that 
\begin{equation}\label{conic1}
\left\{
\begin{array}{ccc}
x_2(s)-x_1(t)&=&\widehat{\lambda}\cdot \left(x_1(t)-\frac{\tilde{a}_1}{\tilde{a}_4}\right)\\
y_2(s)-y_1(t)&=&\widehat{\lambda}\cdot \left(y_1(t)-\frac{\tilde{a}_2}{\tilde{a}_4}\right)\\
z_2(s)-z_1(t)&=&\widehat{\lambda}\cdot \left(z_1(t)-\frac{\tilde{a}_3}{\tilde{a}_4}\right).
\end{array}
\right.
\end{equation}
We will say that two $t,s$ satisfying Eq.\eqref{conic1} for some $\widehat{\lambda}$, are {\it related} by the projection ${\mathcal P}_{\tilde{\bfa}}$; i.e. $t,s$ are related by ${\mathcal P}_{\tilde{\bfa}}$ iff $\bfx_2(s)={\mathcal P}_{\tilde{\bfa}}(\bfx_1(t))$.

Calling $\lambda=\frac{\widehat{\lambda}}{\tilde{a}_4}$, Eq. \eqref{conic1} gives rise to
\begin{equation}\label{conic2}
\left\{
\begin{array}{ccc}
x_2(s)-x_1(t)&=&\lambda\cdot \left(\tilde{a}_4\cdot x_1(t)-\tilde{a}_1\right)\\
y_2(s)-y_1(t)&=&\lambda\cdot \left(\tilde{a}_4\cdot y_1(t)-\tilde{a}_2\right)\\
z_2(s)-z_1(t)&=&\lambda\cdot \left(\tilde{a}_4\cdot z_1(t)-\tilde{a}_3\right) 
\end{array}
\right.
\end{equation}

Now notice that when $\CCC_2$ is the projection of $\CCC_1$ from a projective point $\tilde{\bfa}=[\tilde{a}_1:\tilde{a}_2:\tilde{a}_3:0]$ at infinity, Eq.\eqref{conic2} also holds. Indeed, in that case Eq.\eqref{conic2} expresses that the vector connecting $\bfx_1(t)$ and $\bfx_2(s)$, where $t,s$ are related, is a multiple of the vector ${\bf v}=(-\tilde{a}_1,-\tilde{a}_2,-\tilde{a}_3)$, which defines the projection direction; $\lambda=\lambda(t,s)$ is the proportionality constant between $\bfx_1(t)-\bfx_2(s)$ and $\bfv$.

Therefore, in order to solve the projection problem we need to study Eq. \eqref{conic2}: every solution with $\tilde{a}_4= 0$ corresponds to a parallel projection, and every solution with $\tilde{a}_4\neq 0$ corresponds to a perspective projection. 

The next theorem is crucial. We recall here the definition of \emph{degree} of a rational function: if $\psi(t)=\frac{p(t)}{q(t)}$, where $p(t),q(t)$ are polynomials, the degree of $\psi(t)$ is the maximum of the degrees of $p(t),q(t)$. Furthermore, we define the \emph{associated polynomial} of a given rational function $\psi(t)=\frac{p(t)}{q(t)}$ as $G(t,s)=p(t)-sq(t)$. Finally, we recall that \emph{M\"obius transformation} is a function $\varphi(t)=\frac{at + b}{ct + d}$, with $ad-bc\neq 0$; it is well-known that M\"obius transformations are the birational transformations of the complex line \cite{SWPD}. 

\begin{theorem} \label{essent}
If $\CCC_2={\mathcal P}_{\tilde{\bfa}}({\mathcal C}_1)$, then there exists a real rational function $\psi(t)$ such that ${\mathcal P}_{\tilde{\bfa}}\circ \bfx_1=\bfx_2\circ \psi$. Furthermore, if ${\mathcal P}_{\tilde{\bfa}}$ is non-degenerate, then $\psi(t)$ is a M\"obius transformation. 
\end{theorem}

\begin{proof} Since by hypothesis $\bfx_2$ is proper, $\bfx_2^{-1}$ exists. Therefore, the function $\psi=\bfx_2^{-1}\circ {\mathcal P}_{\tilde{\bfa}} \circ \bfx_1$ (see Diagram \eqref{eq:fundamentaldiagram}) defines a mapping from ${\Bbb C}$ to ${\Bbb C}$ such that the $t$-value generating a point in $\CCC_1$ via $\bfx_1$, is mapped onto the $s$-value generating the point in $\CCC_2$ that is the projection of $\bfx_1(t)$. 
\begin{equation}\label{eq:fundamentaldiagram}
\xymatrix{
\CCC_1 \ar[r]^{{\mathcal P}_{\tilde{\bfa}}} & \CCC_2\ar@{-->}[d]^{\bfx_2^{-1}} \\
\CC \ar@{-->}[u]^{\bfx_1} \ar@{-->}[r]_{\psi} & \CC 
}
\end{equation}
Since $\bfx_2$ is rational then $\bfx_2^{-1}$ is also rational, and since $\bfx_1$ and ${\mathcal P}_{\tilde{\bfa}}$ are rational too, we deduce that $\psi$ is a rational function. Finally, if ${\mathcal P}_{\tilde{\bfa}}$ is non-degenerate then it has an inverse, and therefore $\psi$ also has an inverse, which must be rational. Therefore $\psi$ is a birational transformation of the complex line, and hence it is a M\"obius transformation. 
\end{proof}

We will say that the function $\psi(t)$ is {\it associated} with ${\mathcal P}_{\tilde{\bfa}}$. Additionally, if $G(t,s)$ is the polynomial associated with $\psi(t)$, we will also say that $G(t,s)$ is \emph{associated} with ${\mathcal P}_{\tilde{\bfa}}$. Observe that if $\CCC_2={\mathcal P}_{\tilde{\bfa}}(\CCC_1)$ then the projection of $\bfx_1(t)$ under ${\mathcal P}_{\tilde{\bfa}}$ is the point $\bfx_2(\psi(t))$ (see Figure \ref{fig:cor}). Furthermore, Theorem \ref{essent} provides the following corollary, which follows by substituting $s=\psi(t)$ into Eq. \eqref{conic2}, taking into account that $\bfx_2(s)$ and $\bfx_1(t)$ are rational parametrizations. 

\begin{figure}
\begin{center}
\includegraphics[scale=0.4]{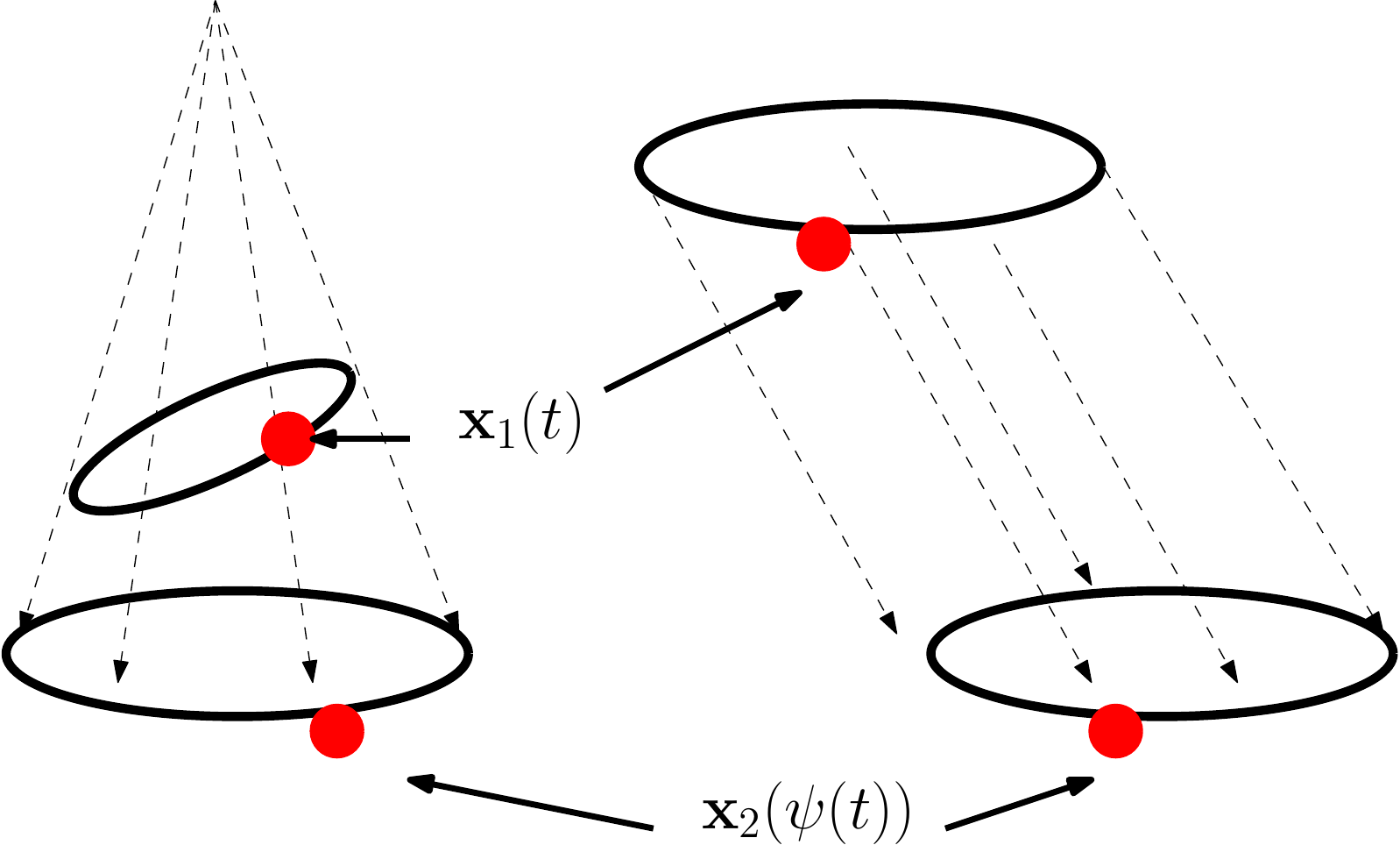} 
\end{center}
\caption{Illustrating Theorem \ref{essent}.}\label{fig:cor}
\end{figure}

\begin{corollary}\label{corfind}
Let $\CCC_2={\mathcal P}_{\tilde{\bfa}}(\CCC_1)$, and let $\psi(t)$ be associated with ${\mathcal P}_{\tilde{\bfa}}$. Then there exists a rational function $\lambda(t)$ such that for each $t$, the values $s=\psi(t)$, $\lambda=\lambda(t)$ satisfy Eq. \eqref{conic2}.
\end{corollary}

\begin{remark}\label{distinct}
If there are several projections ${\mathcal P}_{\tilde{\bfa}_i}$ such that $\CCC_2={\mathcal P}_{\tilde{\bfa}_i}(\CCC_1)$, $i=1,\ldots,m$, we can see from the diagram \eqref{eq:fundamentaldiagram} that each one must correspond to a different $\psi_i(t)$. 
\end{remark}

In order to find the projection ${\mathcal P}_{\tilde{\bfa}}$, if any, we will focus on computing the associated function $\psi(t)$ described in Theorem \ref{essent}. More precisely, we will seek the polynomial associated with $\psi(t)$. Now let
\begin{equation}\label{eqn}
N(t,s)=\frac{\partial \bfx_1(t)}{\partial t}\times \frac{\partial \bfx_2(s)}{\partial s}.
\end{equation}
Then we have the following result.

\begin{proposition}\label{asoc}
If $t,s$ are related by the projection ${\mathcal P}_{\tilde{\bfa}}$, then 
\begin{equation}\label{condi}
N(t,s)\cdot (\bfx_1(t)-\bfx_2(s))=0,
\end{equation}
where $\cdot$ represents the usual Euclidean dot product. 
\end{proposition}

\begin{proof} If $\CCC_2={\mathcal P}_{\tilde{\bfa}}(\CCC_1)$ then $\CCC_1$ and $\CCC_2$ are contained in a developable surface $S$, which is cylindrical when $\tilde{\bfa}$ is at infinity, and conical when $\tilde{\bfa}$ is not. Furthermore, if $t,s$ are related, then there is a generatrix of $S$ connecting $\bfx_1(t)$ and $\bfx_2(s)$. At 
each point $\bfp=\bfx_1(t)$ the vector $\frac{\partial \bfx_1(t)}{\partial t}$ belongs to the tangent plane to $S$ at $\bfp$. Similarly, at each point $\bfq=\bfx_2(s)$ the vector $\frac{\partial \bfx_2(s)}{\partial s}$ belongs to the tangent plane to $S$ at $\bfq$. Since the tangent plane at a point of a developable surface contains the generatrix through the point \cite[\S 2.4]{Struik}, if $t,s$ are related then $\bfx_1(t)-\bfx_2(s)$, $\frac{\partial \bfx_1(t)}{\partial t}$, $\frac{\partial \bfx_2(s)}{\partial s}$ belong to the same plane. Therefore, they are linearly dependent, and Condition \eqref{condi} follows.
\end{proof} 

Proposition \ref{asoc} provides the following corollary; we denote the square-free part of the numerator of the rational function at the left hand-side of \eqref{condi}, by ${\mathcal N}(t,s)$.

\begin{corollary}\label{cor1}
If $t,s$ are related under the projection ${\mathcal P}_{\tilde{\bfa}}$, and $G(t,s)$ is associated with ${\mathcal P}_{\tilde{\bfa}}$, then $G(t,s)$ divides ${\mathcal N}(t,s)$.
\end{corollary}

Corollary \eqref{cor1} suggests a method for computing the factors $G(t,s)$ giving rise to projections: (a) compute the factors of ${\mathcal N}(t,s)$ (by means of factoring techniques); (b) pick the factors which are linear in $s$, and compute the corresponding functions $\psi(t)$; (c) find the projections by making use of Eq. \eqref{conic2}. Observe that this strategy requires that ${\mathcal N}(t,s)$ is not identically zero. 

\begin{lemma}\label{instru}
Under the preceding hypotheses, ${\mathcal N}(t,s)$ cannot be identically zero.
\end{lemma} 

\begin{proof} We argue by contradiction. Suppose that $N(t,s)\cdot (\bfx_1(t)-\bfx_2(s))$ is identically zero. Then for all $(t,s)\in {\Bbb C}^2$,
\begin{equation}\label{det}
\left|\begin{array}{ccc}
x_1(t)-x_2(s)&y_1(t)-y_2(s)&z_1(t)-z_2(s)\\x'_1(t)&y'_1(t)&z'_1(t)\\ x_2'(s)&y_2'(s)&z_2'(s)
\end{array}\right|=0.
\end{equation}
We will prove that if the above determinant is zero, $\CCC_1$, $\CCC_2$ are contained in a same plane, which was excluded at the beginning of the paper. In order to do this, we will first see that $\CCC_1,\CCC_2$ must both be planar, and then we will prove that the plane containing both curves must coincide. We start proving that $\CCC_1$ is planar. Let $s=s_0$ satisfy that $\bfx_2(s_0)$ is well-defined, and $\bfx_2'(s_0)=(u_0,v_0,w_0)\neq 0$. By substituting $s=s_0$ in Eq. \eqref{det}, we get 
\begin{equation}\label{det2}
\left|\begin{array}{ccc}
x_1(t)-x_0 & y_1(t)-y_0 & z_1(t)-z_0 \\ x'_1(t) & y'_1(t) & z'_1(t)\\ u_0 & v_0 & w_0 \end{array}\right|=0
\end{equation}
for all $t$. Since $\bfx_2'(s_0)=(u_0,v_0,w_0)\neq 0$ we can assume that, say, $w_0$, is nonzero. Then there exist $\gamma,\mu$ such that $u_0=\gamma \cdot w_0,v_0=\mu \cdot w_0$. Now perform the following elementary transformations in the determinant at the left hand-side of Eq. \eqref{det2}: multiply the last column by $-\gamma$, and add it up to the first column; multiply the last column by $-\mu$, and add it up to the second column. Hence, we get
\begin{equation}\label{det3}
\left|\begin{array}{ccc}
x_1(t) -\gamma \cdot z_1(t)-x_0+\gamma \cdot z_0 & y_1(t) - \mu \cdot z_1(t)-y_0+\mu \cdot z_0 & z_1(t)-z_0 \\ x'_1(t) - \gamma \cdot z'_1(t) & y'_1(t) -\mu \cdot z'_1(t) & z'_1(t)\\ 0 & 0 & w_0 \end{array}\right|=0.
\end{equation}
Calling $X(t)=x_1(t)-\gamma \cdot z_1(t)-x_0+\gamma \cdot z_0$, $Y(t)=y_1(t)-\mu \cdot z_1(t)-y_0+\mu \cdot z_0$ and $Z(t)=z_1(t)-z_0$, we have 
\begin{equation}\label{det4}
\left|\begin{array}{ccc}
X(t) & Y(t)& Z(t) \\ X'(t) & Y'(t) & Z'(t)\\ 0 & 0 & w_0 \end{array}\right|=0.
\end{equation}
If $X(t)$ is identically zero then $\CCC_1$ is contained in the plane of equation $x-\gamma z-x_0+\gamma \cdot z_0=0$; similarly, if $Y(t)$ is identically zero then $\CCC_1$ is contained in the plane of equation $y-\mu z-y_0+\mu \cdot z_0=0$. If either $X'(t)$ or $Y'(t)$ are identically zero, then we can also easily conclude that $\CCC_1$ is planar. In any other case, by expanding the determinant at the left hand-side of Eq. \eqref{det4}, we get $X(t)Y'(t)-X'(t)Y(t)=0$, which leads to  $X(t)=\nu \cdot Y(t)$, with $\nu$ a constant. Hence, we deduce that $\CCC_1$ is contained in the plane 
\[x-\nu y-(\gamma-\mu\nu)z-x_0+\nu y_0=0.\]
Let $\Pi_1$ be the plane containing $\CCC_1$, and let $\Pi_2$ be the plane containing $\CCC_2$. Let $P=\bfx_1(t_0)\in \CCC_1$, where the tangent vector $\vec{v}_P=\bfx'_1(t_0)$ to $\CCC_1$ is well-defined, and let $Q=\bfx_2(s_0)\in \CCC_2$, where the tangent vector $\vec{v}_Q$ to $\CCC_2$ is also well-defined. Since $\CCC_1,\CCC_2$ are not both straight lines, we can assume that $\vec{v}_P$ and $\vec{v}_Q$ are not parallel. If $\Pi_1\neq \Pi_2$ then 
$\{\vec{PQ},\vec{v}_P,\vec{v}_Q\}$ are linearly independent. Therefore $\det\left(\vec{PQ},\vec{v}_P,\vec{v}_Q\right)\neq 0$, i.e. $\det\left(\bfx_1(t_0)-\bfx_2(s_0),\bfx'_1(t_0),\bfx'_2(s_0)\right)\neq 0$. However this is contradictory with Eq. \eqref{det}; hence, $\Pi_1=\Pi_2$, which is excluded by hypothesis. 
\end{proof}

Since ${\mathcal N}(t,s)$ is not identically zero, ${\mathcal N}(t,s)$ has finitely many factors; so taking Remark \ref{distinct} into account, Lemma \ref{instru} provides the following corollary.

\begin{corollary}\label{cor2}
The number of projections ${\mathcal P}_{\tilde{\bfa}}$ such that $\CCC_2={\mathcal P}_{\tilde{\bfa}}(\CCC_1)$, is finite.
\end{corollary}

We need to consider an additional question. If there are two projections ${\mathcal P}_{\tilde{\bfa}_1}\neq {\mathcal P}_{\tilde{\bfa}_2}$ mapping $\CCC_1$ onto $\CCC_2$, it might exist points $\bfp\in \CCC_1$ and $\bfq\in \CCC_2$ that are mapped onto each other by \emph{both} projections, i.e. such that 
\[
\bfq={\mathcal P}_{\tilde{\bfa}_1}(\bfp)={\mathcal P}_{\tilde{\bfa}_2}(\bfp).\]
For instance, consider the two lemniscatas shown in Fig. \ref{fig:unlucky}. These two lemniscatas are mapped onto each other by two different projections, as shown by the picture. However, both projections map the singular point of $\CCC_1$ onto the singular point of $\CCC_2$. 

\begin{figure}
\begin{center}
$$\begin{array}{cc}
\includegraphics[scale=0.3]{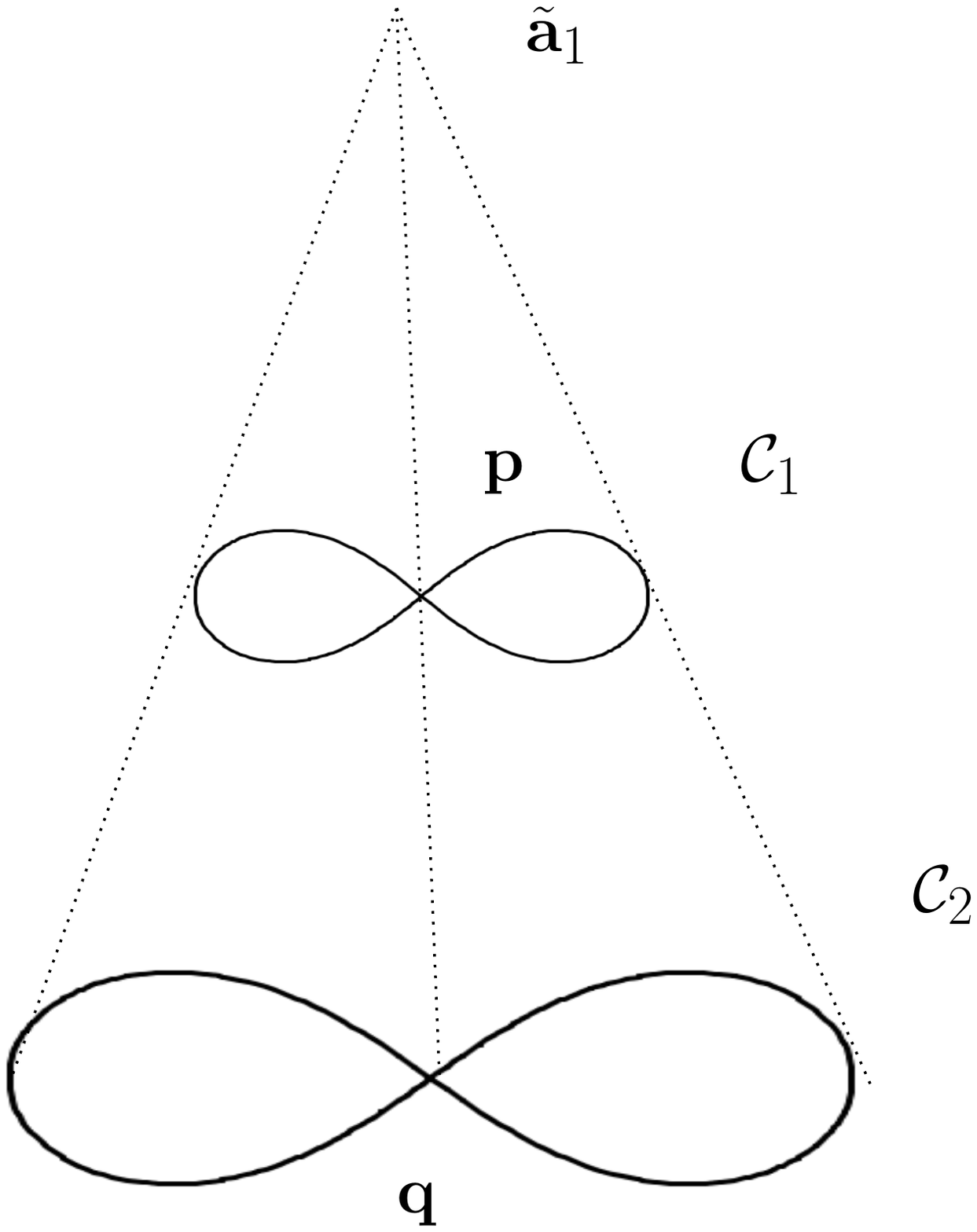} &  \includegraphics[scale=0.3]{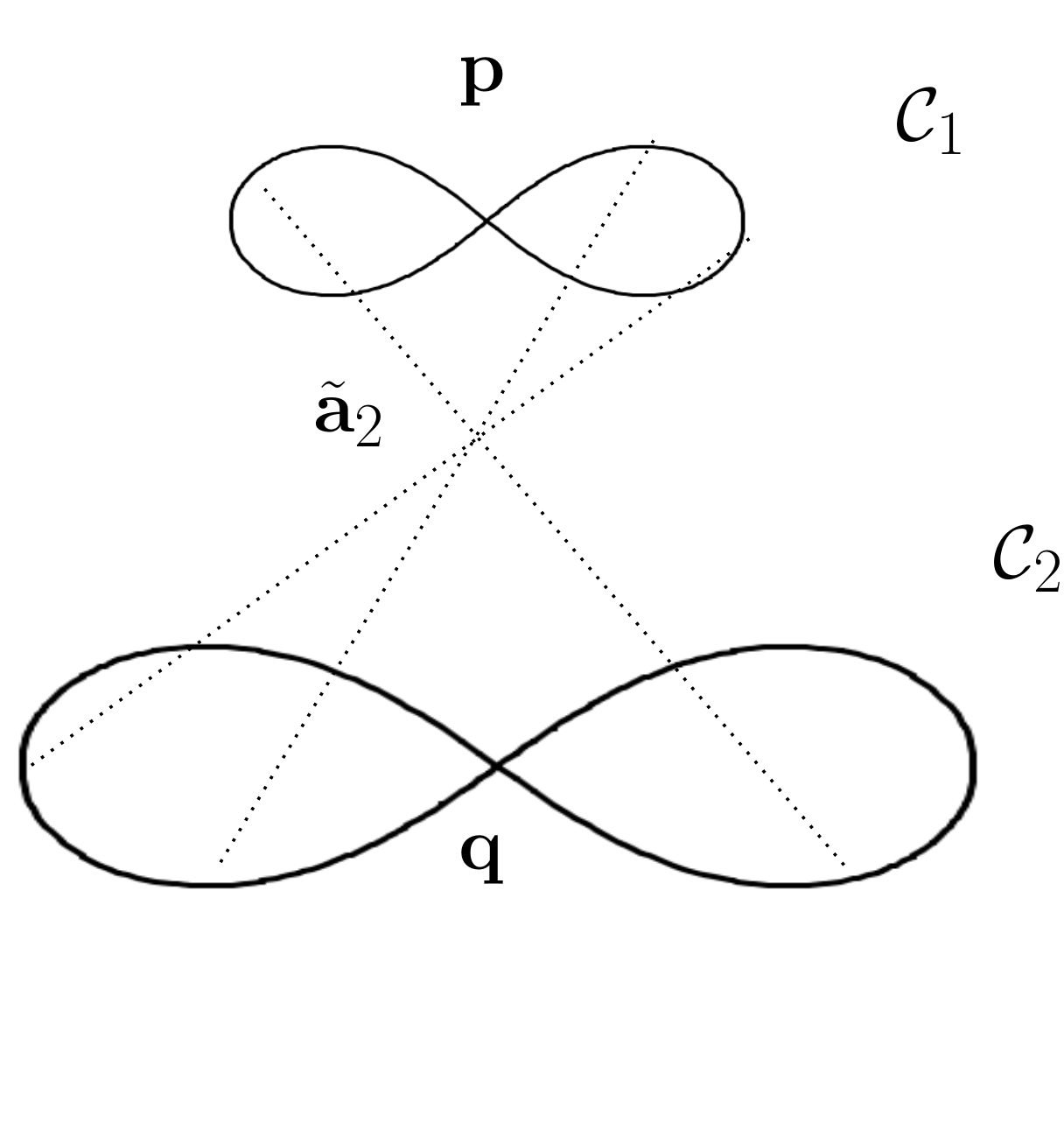}
\end{array}$$
\end{center}
\caption{Unlucky points.}\label{fig:unlucky}
\end{figure}

For these pairs $(\bfp,\bfq)$, the line connecting $\bfp,\bfq$ contains both eye points. The next lemma shows that we just have finitely pairs of this kind, that can be computed. Here we represent by ${\mathcal N}^{\star}$ the planar algebraic curve, in the $(t,s)$ plane, defined by the polynomial ${\mathcal N}(t,s)$. 

\begin{lemma}\label{badpoints}
Let $\tilde{\bfa}_1\neq \tilde{\bfa}_2$ such that ${\mathcal P}_{\tilde{\bfa}_1}$ and ${\mathcal P}_{\tilde{\bfa}_2}$ map $\CCC_1$ onto $\CCC_2$, and let $\bfp\in \CCC_1$ satisfy that 
${\mathcal P}_{\tilde{\bfa}_1}(\bfp)={\mathcal P}_{\tilde{\bfa}_2}(\bfp)$. Let $\bfq={\mathcal P}_{\tilde{\bfa}_i}(\bfp)$, for $i=1,2$. Then one of the following situations occur: (i) $\bfp=P_1^{\infty}$; (ii) $\bfq=P_2^{\infty}$; (iii) $\bfq$ is a self-intersection of $\CCC_2$; (iv) $\bfp=\bfx_1(t_0)$, where the line $t=t_0$ contains a singular point of ${\mathcal N}^{\star}$.
\end{lemma}

\begin{proof} Let $\psi_1(t),\psi_2(t)$ be associated with ${\mathcal P}_{\tilde{\bfa}_1}$, ${\mathcal P}_{\tilde{\bfa}_2}$. Assume that there exists $t_0\in {\Bbb C}$ such that $\bfp=\bfx_1(t_0)$, and assume also that $\psi_1(t_0)$, $\psi_2(t_0)$ are well defined. Notice that if any of these conditions is not satisfied, then either $\bfp=P_1^{\infty}$ or $\bfq=P_2^{\infty}$. Since ${\mathcal P}_{\tilde{\bfa}_1}(\bfx_1(t_0))={\mathcal P}_{\tilde{\bfa}_2}(\bfx_1(t_0))$ and ${\mathcal P}_{\tilde{\bfa}_i}(\bfx_1(t_0))=\bfx_2(\psi_i(t_0))$ for $i=1,2$, we have 
\[\bfx_2(\psi_1(t_0))=\bfx_2(\psi_2(t_0)).\]Now if $\psi_1(t_0)\neq \psi_2(t_0)$, then $\bfq$ is a self-intersection of $\CCC_2$. If $\psi_1(t_0)=\psi_2(t_0)$, then the point $(t_0,s_0)$, where $s_0=\psi_1(t_0)=\psi_2(t_0)$, is a point of ${\mathcal N}^{\star}$ where two different branches of ${\mathcal N}^{\star}$ (namely, the ones corresponding to the functions $s=\psi_1(t)$ and $s=\psi_2(t)$) intersect. So $(t_0,s_0)$ is a singularity of ${\mathcal N}^{\star}$.
\end{proof}

Lemma \ref{badpoints} gives rise to the following definition.

\begin{definition}\label{luck} Let $(\bfp,\bfq)$, where $\bfp\in \CCC_1$, $\bfq\in \CCC_2$ and ${\mathcal P}_{\tilde{\bfa}}(\bfp)=\bfq$. We say that the pair $(\bfp,\bfq)$ is \emph{lucky}, if $\bfp \neq \bfq$ and $\bfp,\bfq$ do not satisfy any of the conditions (i)-(iv) in Lemma \ref{badpoints}. Otherwise, we say that $(\bfp,\bfq)$ is \emph{unlucky}.
\end{definition}

Notice that the set of unlucky pairs is finite. Observe that a priori we do not know the number of projections mapping $\CCC_1$ onto $\CCC_2$. Hence, the unlucky pairs correspond to the  \emph{potential} points in $\CCC_1$ and $\CCC_2$ such that, if there are several projections mapping $\CCC_1$ onto $\CCC_2$, are mapped onto its partner by all those projections. Unlucky pairs, as we will see later, are to be avoided, because we will be interested in pairs of points such that the line connecting them contains just one eye point, if any. 

In the rest of the paper, given a function $\psi(t)$, associated with some projection, we denote by ${\mathcal G}_{\psi}$ the set of (finitely many) $t_0$ values such that the pair $(\bfp,\bfq)$, where $\bfp=\bfx_1(t_0)$ and $\bfq=\bfx_2(\psi(t_0))$, is unlucky.

\subsection{A factoring-based algorithm.}\label{sec-algo}

As we observed after Corollary \ref{cor1}, in order to solve our problem, once we compute the factors of ${\mathcal N}(t,s)$ which are linear in $s$ we just need to find out which of these factors, if any, correspond to projections mapping $\CCC_1$ onto $\CCC_2$. In order to do so, the notion of \emph{lucky point} is essential. Let $G(t,s)$ be a factor of ${\mathcal N}(t,s)$, linear in $s$, and let $s=\psi(t)$ be the result of solving for $s$ in $G(t,s)=0$. By Lemma \ref{badpoints}, given $t_1\neq t_2$, $t_1,t_2\notin {\mathcal G}_{\psi}$ there is just one projection ${\mathcal P}_{\tilde{\bfa}}$ mapping $\bfx_1(t_j)\to \bfx_2(\psi(t_j))$, $j=1,2$. Furthermore, if $t_1\neq t_2$, $t_1,t_2\notin {\mathcal G}_{\psi}$ there is a line ${\mathcal L}_j$ connecting $\bfx_1(t_j)$ and $\bfx_2(\psi(t_j))$, $j=1,2$, and $\tilde{\bfa}$ is the (projective) intersection point of ${\mathcal L}_1$ and ${\mathcal L}_2$. Then we have the following algorithm, {\tt Algorithm 1}, to solve the projection problem.

\begin{algorithm}[h!]
\begin{algorithmic}[1]
\REQUIRE Two proper, rational parametrizations $\bfx_1(t),\bfx_2(s)$ defining two space curves $\CCC_1,\CCC_2$, where $\CCC_2$ is planar.
\ENSURE The projections, if any, mapping $\CCC_1$ onto $\CCC_2$.
\STATE compute the polynomial ${\mathcal N}(t,s)$.
\STATE find the real factors $G_i(t,s)$ of ${\mathcal N}(t,s)$ which are linear in $s$. 
\FOR{each $G_i(t,s)$}
\STATE{solve $G_i(t,s)$ for $s$, to find $s=\psi_i(t)$.}
\STATE{pick $t_1\neq t_2$, $t_1,t_2\notin {\mathcal G}_{\psi_i}$.} 
\STATE{compute the lines ${\mathcal L}_j$, $j=1,2$, connecting $\bfx_1(t_j)$ and $\bfx_2(\psi(t_j))$.}
\IF{${\mathcal L}_1\cap {\mathcal L}_2$ intersect at an affine point $(c_1,c_2,c_3)$,}
\STATE{$\tilde{\bfa}_i:=[c_1:c_2:c_3:1]$}
\ELSE
\STATE{$\tilde{\bfa}_i:=[v_1:v_2:v_3:0]$, where ${\bf v}=(v_1,v_2,v_3)$ is parallel to ${\mathcal L}_1$ and ${\mathcal L}_2$.}
\ENDIF
\STATE{plug $\tilde{\bfa}_i$ into Eq. \eqref{conic2}, and check if 
\[\displaystyle{\frac{x_2(\psi_i(t))-x_1(t)}{\tilde{a}_4\cdot x_1(t)-\tilde{a}_1}=\frac{y_2(\psi_i(t))-y_1(t)}{\tilde{a}_4\cdot y_1(t)-\tilde{a}_2}=\frac{z_2(\psi_i(t))-z_1(t)}{\tilde{a}_4\cdot z_1(t)-\tilde{a}_3}}.\]}
\STATE{in the affirmative case, add ${\mathcal P}_{\tilde{\bfa}_i}$ to the list of projections (initially empty).}
\ENDFOR
\STATE{{\bf return} the list of projections mapping $\CCC_1$ to $\CCC_2$, or the message {\tt No projection has been found}.}
\end{algorithmic}
\caption*{{\bf Algorithm} 1}
\end{algorithm}

Algorithm 1 is illustrated in the following example. In this example and in the other examples of the paper, the computations have been done with Maple 18, running on a laptop, with 2.9 GHz i7-3520M processor and 8 Gb RAM.

\begin{example}
Let $\CCC_1$ be the curve parametrized by $\bfx_1(t)=(x_1(t),y_1(t),z_1(t))$, where 
\[
\begin{array}{ccc}
x_1(t) & = & \frac{2t^5+11t^4+14t^3-3t^2-12t+4}{(t+2)^2},\\
y_1(t) & = & \frac{-2t^2(t^3+3t^2+3t-2)}{t+2},\\
z_1(t) & = & \frac{2t^6+10t^5+19t^4+14t^3-3t^2-12t+4}{(t+2)^2},
\end{array}
\]
and let $\CCC_2$ be the curve parametrized by
\[\bfx_2(s)=(10s+5,-2(s+3)(s-2),2s^2+s+2).\]One can check that $\CCC_2$ is planar, and that both $\CCC_1$ and $\CCC_2$ are contained in the cone $x^2+y^2-z^2=0$; therefore, they must be related by at least one perspective projection from the origin, which is the vertex of the cone. In this case, by factoring over the rationals, we get 
\[
\begin{array}{ccl}
{\mathcal N}(t,s)&=&-8(t^6+9t^5+5st^3-34t^4+30st^2-330t^3+80ts-684t^2+80s\\
&&-544t-160)\cdot (st^2-5t^3+3ts-12t^2-2s-6t+4).
\end{array}
\]
Therefore, we have
\[\psi_1(t)=\frac{5t^3+12t^2+6t-4}{t^2+3t-2},\mbox{ }\psi_2(t)=\frac{-t^6-9t^5+34t^4+330t^3+684t^2+544t+160}{5(t^3+6t^2+16t+16)}.\]
In order to check if $\psi_1(t)$ corresponds to some projection, we observe first that $t_1=0$, $t_2=1$ are not in ${\mathcal G}_{\psi_1}$. The lines ${\mathcal L}_i$, $i=1,2$, connecting the points $\bfx_1(t_i)$ and $\bfx_2(\psi_1(t_i))$, intersect at the origin. The test in Step 12 of Algorithm 1 is positive, so we conclude that $\CCC_1$ is mapped onto $\CCC_2$ by a perspective projection from the origin, as expected. Furthermore, since $\psi_1(t)$ has degree 3, the projection is degenerate. Now we proceed in the same way with $\psi_2(t)$. In this case we also take $t_1=0$, $t_2=1$; however, the corresponding lines ${\mathcal L}_i$, $i=1,2$ are skew, so $\psi_2(t)$ does not correspond to any projection. The whole computation takes 0.062 seconds. 
\end{example}

If all the factors of ${\mathcal N}(t,s)$ which are linear in $s$ have rational coefficients, then we can find them by applying standard bivariate factoring algorithms over the rationals, implemented in most computer algebra systems. The next result shows that, under the hypothesis that $\bfx_1(t)$ and $\bfx_2(s)$ are parametrizations with rational coefficients, projections from \emph{rational} eye points can always be found this way.

\begin{proposition} \label{rati}
If $\tilde{\bfa}$ has rational coordinates, the polynomial $G(t,s)$ associated with ${\mathcal P}_{\tilde{\bfa}}$ has rational coefficients.
\end{proposition}

\begin{proof} Since by hypothesis $\bfx_2(s)$ has rational coefficients, the projection plane $\Pi$, which is the plane containing the curve parametrized by $\bfx_2(s)$, has rational coefficients too. Since also by hypothesis $\tilde{\bfa}$ has rational coefficients, the elements of the matrix $P$ associated with ${\mathcal P}_{\tilde{\bfa}}$ are rational (see Eq. \eqref{matP}). Therefore, from Diagram \eqref{eq:fundamentaldiagram}, $\psi$ has rational coefficients. 
\end{proof}

However, as we saw in Section \ref{sec-state}, we can have projections mapping $\CCC_1$ onto $\CCC_2$ from a point with non-rational coefficients. In order to also find these projections, we need to factor ${\mathcal N}(t,s)$ over the reals, i.e. one needs an {\it absolute} factorization of ${\mathcal N}(t,s)$ \cite{CGKW}, \cite{Gao}, \cite{Kal08}. This is implemented, for instance, in the computer algebra system Maple 18 through the command {\tt AFactors}, and works finely for moderate and medium degrees. 


\begin{example}
Let $\CCC_1$ be parametrized by $\bfx_1(t)=\left(\frac{p_1(t)}{p_4(t)},\frac{p_2(t)}{p_4(t)},\frac{p_3(t)}{p_4(t)}\right)$, where
\[
\begin{array}{ccl}
p_1(t) & = & 2t^6+t^3-t^2-3,\\
p_2(t) & = & -t^{11}-t^9-3t^8-2t^5+t^3,\\
p_3(t) & = & -3t^{11}+t^5-t^3-3t^2+2t+1,\\
p_4(t) & = & -2t^{10}-t^8-3t^6-3t^5-t+1.
\end{array}
\]Also, let $\CCC_2$ be parametrized by $\bfx_2(s)=\left(\frac{q_1(s)}{q_4(s)},\frac{q_2(s)}{q_4(s)},\frac{q_3(s)}{q_4(s)}\right)$, where
\[
\begin{array}{ccl}
q_1(s) & = & -4s^{11}-s^9-3s^8-s^5-3s^2+2s+1,\\
q_2(s) & = & -3s^{11}+2s^6+s^5-4s^2+2s-2,\\
q_3(s) & = & 7s^{11}+s^9+3s^8-2s^6+7s^2-4s+1,\\
q_4(s) & = & -4s^{11}+2s^{10}-s^9-2s^8+5s^6+2s^5+s^3-4s^2+3s-3.
\end{array}
\]
One can check that $\CCC_2$ is the image of $\CCC_1$ under the projection defined by the matrix 
\[
P=\begin{bmatrix} 
0 & 1 & 1 & 0 \\
1 & 0 & 1 & 0 \\
-1 & -1 & 2 & 0 \\
1 & 1 & 1 & -1
\end{bmatrix}.
\]
An absolute factorization of ${\mathcal N}(t,s)$ yields ${\mathcal N}(t,s)={\mathcal N}_1(t,s)\cdot {\mathcal N}_2(t,s)$. The first factor is ${\mathcal N}_1(t,s)=st+s+t-1$. 
Therefore, it gives rise to
\[\psi_1(t)=-\frac{t-1}{t+1}.\]One can check that $t_1=3$, $t_2=4$ satisfy that $t_1,t_2\notin {\mathcal G}_{\psi_1}$. Then we compute the line ${\mathcal L}_1$, connecting the points $\bfx_1(3),\bfx_2(\psi_1(3))$, and the line ${\mathcal L}_2$, connecting the points $\bfx_1(4),\bfx_2(\psi_1(4))$. These two lines intersect at the affine point $(1,1,-1)$. Furthermore, the test in Step 12 of Algorithm 1 is positive, so we conclude that $\CCC_2$ is the perspective projection of $\CCC_1$ from the point $(1,1,-1)$.  On the other hand, the second factor ${\mathcal N}_2(t,s)$ is a big polynomial, of total degree equal to 38, and infinity norm equal to 319426480. Since the degree in $s$ of ${\mathcal N}_2(t,s)$ is equal to 19, therefore different from 1, ${\mathcal N}_2(t,s)$ does not correspond to any projection. The whole computation takes 3.526 seconds. 
\end{example} 

\section{An alternative method.} \label{sec-alter}

In this section we present another method for computing the projections that does not require absolute factoring, and that, additionally, can be adapted to curves which are known up to finite precision, i.e. whose parametrizations have floating point coefficients. The method uses the fact that we have a rational transformation $\psi$ behind each projection, as we saw in the previous section, without actually computing it. 

We first provide a symbolic version of the method. In general, this symbolic method requires to compute the primitive element of an algebraic extension, and then carry out computations in the new simple extension. However, these operations can be time-consuming. Therefore in practice this symbolic method is only useful when $\CCC_1$ and $\CCC_2$ are low-degree curves. However, one can also have an \emph{approximate} version of the method, which is really fast. 

In the symbolic version of the method, we first find tentative values $\tilde{\bfa}_i$ for the eye points of the projections mapping $\CCC_1$ onto $\CCC_2$, if any; then, for each $\tilde{\bfa}_i$, we check whether ${\mathcal P}_{\tilde{\bfa}_i}(\CCC_1)=\CCC_2$. In the approximate version, we compute approximations of the $\tilde{\bfa}_i$, and then we evaluate ``how close" ${\mathcal P}_{\tilde{\bfa}_i}(\CCC_1)$ and $\CCC_2$ are, i.e. we check whether ${\mathcal P}_{\tilde{\bfa}_i}(\CCC_1) \approx \CCC_2$, instead of ${\mathcal P}_{\tilde{\bfa}_i}(\CCC_1)=\CCC_2$.

\subsection{Symbolic version of the method.} \label{subsec-symb}

The first step is to find tentative values $\tilde{\bfa}_i$ for the eye points of the projections mapping $\CCC_1$ onto $\CCC_2$, if any. In order to do this, we need to compute two different lucky pairs, $(\bfp_1,\bfq_1)$ and $(\bfp_2,\bfq_2)$, which are related by the {\it same} projection.  

\begin{lemma}\label{thesame}
Let $(\bfp_1,\bfq_1)$ and $(\bfp_2,\bfq_2)$ be two different lucky pairs such that $\bfp_j\in \CCC_1$, $\bfq_j\in \CCC_2$, $\bfp_j=\bfx_1(t_j)$, $\bfq_j=\bfx_2(s_j)$, $t_j,s_j\in {\Bbb R}$ for $j=1,2$, and the points $(t_1,s_1)$ and $(t_2,s_2)$ belong to ${\mathcal N}^{\star}$. If $(t_1,s_1)$ and $(t_2,s_2)$ are connected by one real branch of the curve ${\mathcal N}^{\star}$, such that there is no point $(\widehat{t},\widehat{s})\in {\mathcal N}^{\star}$, with $\widehat{t}\in [t_1,t_2]$, where the partial derivative $\frac{\partial {\mathcal N}}{\partial t}$ vanishes, then there is at most one projection ${\mathcal P}_{\tilde{\bfa}}$ mapping $\CCC_1$ onto $\CCC_2$, such that ${\mathcal P}_{\tilde{\bfa}}(\bfp_j)=\bfq_j$ for both $j=1,2$.
\end{lemma}

\begin{proof} If $(t_1,s_1)$ and $(t_2,s_2)$ are connected by a real branch of ${\mathcal N}^{\star}$ in the conditions of the lemma, then there is at most one rational function $s=\psi(t)$, whose graph is contained in ${\mathcal N}^{\star}$, connecting $(t_1,s_1)$ and $(t_2,s_2)$. Since different projections correspond to different $\psi(t)$s, the result follows.
\end{proof}

If $(\bfp_1,\bfq_1)$ and $(\bfp_2,\bfq_2)$ satisfy the hypotheses of Lemma \ref{thesame}, we will say that the corresponding points $(t_1,s_1)$, $(t_2,s_2)$ of ${\mathcal N}^{\star}$ are \emph{compatible}. This notion, together with Lemma \ref{thesame}, is illustrated in Figure \ref{fig:compat}.

\begin{figure}
\begin{center}
\includegraphics[scale=0.4]{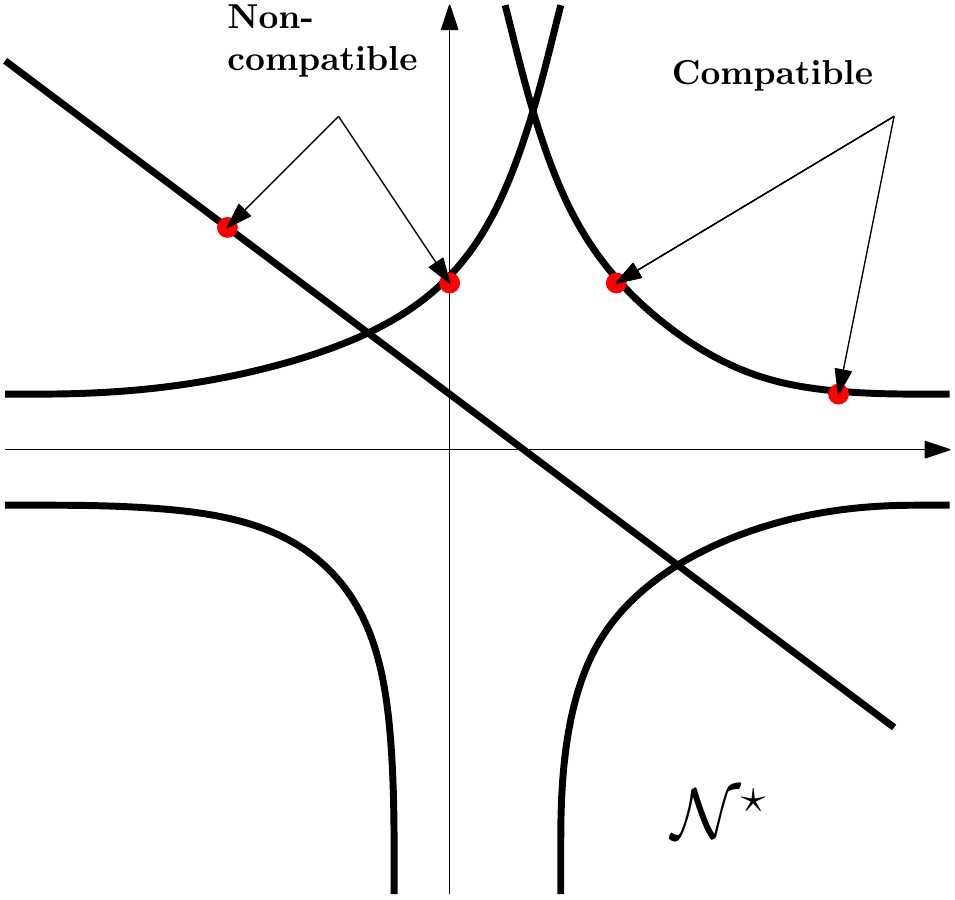} 
\end{center}
\caption{Illustrating Lemma \ref{thesame}: }\label{fig:compat}
\end{figure}

Now we proceed in the following way. Consider a line $t=t_1$, which is not a vertical asymptote of ${\mathcal N}^{\star}$, and not containing any point of ${\mathcal N}^{\star}$ where the partial derivative $\frac{\partial {\mathcal N}}{\partial t}$ vanishes. Notice that since ${\mathcal N}(t,s)$ is square-free, there are just finitely many points of ${\mathcal N}^{\star}$ where $\frac{\partial {\mathcal N}}{\partial t}$ is zero. Therefore, the line $t=t_1$ intersects all the graphs of the different functions $s=\psi(t)$ associated with the projections mapping $\CCC_1$ onto $\CCC_2$. So let $s_{1,1}<\cdots<s_{1,m}$ be the real roots of $n_1(s)={\mathcal N}(t_1,s)$. For each $s_{1,i}$, by Lemma \ref{badpoints} there is just one projection, if any, mapping $\bfx_1(t_1)\to \bfx_2(s_{1,i})$. Furthermore, taking $t_2>t_1$ such that there is no $t$-value in $[t_1,t_2]$ 
corresponding to a vertical asymptote of ${\mathcal N}^{\star}$, or a point of ${\mathcal N}^{\star}$ where $\frac{\partial {\mathcal N}}{\partial t}$ vanishes, $n_2(s)={\mathcal N}(t_2,s)$ also has $m$ real roots $s_{2,1}<\cdots<s_{2,m}$. Therefore, for each $i=1,\ldots,m$ the points $(t_1,s_{1,i})$, $(t_2,s_{2,i})$ are compatible, so they correspond to the same projection ${\mathcal P}_{\tilde{\bfa}_i}$, if any. Hence, by intersecting the lines ${\mathcal L}_i$ and $\tilde{\mathcal L}_i$, connecting the points $\bfx_1(t_1),\bfx_2(s_{1,i})$ and $\bfx_1(t_2),\bfx_2(s_{2,i})$, respectively, we get tentative values for the $\tilde{\bfa}_i$. 

If $s_{1,i}$ and $s_{2,i}$ are rational then we can compute $\tilde{\bfa}_i$ exactly. Otherwise we need to work in an algebraic extension. However, in general we will have two different algebraic numbers $\beta,\gamma$ for $s_{1,i}$ and $s_{2,i}$ respectively, so we need to compute the primitive element $\alpha$ of a double field extension ${\Bbb Q}(\beta,\gamma)$; an algorithm to do so can be found in page 145 of \cite{winkler}. Then the intersection point ${\mathcal L}_i\cap \tilde{\mathcal L}_i$ can be determined doing computations in ${\Bbb Q}(\alpha)$. 

Once the tentative $\tilde{\bfa}_i$ have been computed, we need to check, for each $i$, whether ${\mathcal P}_{\tilde{\bfa}_i}(\CCC_1)=\CCC_2$. In order to do this, let \[\widehat{\bfx}_1(t)=(\widehat{x}_1(t),\widehat{y}_1(t),\widehat{z}_1(t))={\mathcal P}_{\tilde{\bfa}_i}(\bfx_1(t)).\]Since ${\mathcal P}_{\tilde{\bfa}_i}(\CCC_1)$ and $\CCC_2$ are rational they are irreducible, so ${\mathcal P}_{\tilde{\bfa}_i}(\CCC_1)$ and $\CCC_2$ are equal iff they have infinitely many points in common. In order to check this, let us denote the numerators of $\widehat{x}_1(t)-x_2(s)$, $\widehat{y}_1(t)-y_2(s)$, $\widehat{z}_1(t)-z_2(s)$ by $n_1(t,s)$, $n_2(t,s)$, $n_3(t,s)$. Then  ${\mathcal P}_{\tilde{\bfa}_i}(\CCC_1)=\CCC_2$ iff $\gcd(n_1(t,s),n_2(t,s),n_3(t,s))\neq 1$.

So we get the following algorithm, {\tt Algorithm 2}, for solving the projection problem. Algorithm 2 is illustrated in a toy example, Example \ref{Ex2}.

\begin{algorithm}[h!]
\begin{algorithmic}[1]
\REQUIRE Two proper, rational parametrizations $\bfx_1(t),\bfx_2(s)$ defining two space curves $\CCC_1,\CCC_2$, where $\CCC_2$ is planar.
\ENSURE The projections, if any, mapping $\CCC_1$ onto $\CCC_2$.
\STATE compute the polynomial ${\mathcal N}(t,s)$ as the numerator of the left hand-side of Eq. \eqref{condi}.
\STATE pick $t_1,t_2\in {\Bbb Q}$ such that $t=t_1$, $t=t_2$ are not asymptotes of ${\mathcal N}^{\star}$, and such that no singularity of ${\mathcal N}^{\star}$ lies in $[t_1,t_2]\times {\Bbb R}$.
\STATE Let $s_{j,1},\ldots,s_{j,m}$, $j=1,2$, be the real roots of $n_j(s)={\mathcal N}(t_j,s)$.
\FOR{$i=1,\ldots,m$}
\STATE{check if the pair $(\bfx_1(t_1),\bfx_2(s_{1,i}))$, is lucky; in the negative case, pick a different value for $t_1$.}
\STATE{check if the pair $(\bfx_1(t_2),\bfx_2(s_{2,i}))$, is lucky; in the negative case, pick a different value for $t_2$.}
\STATE{compute the line ${\mathcal L}_i$ connecting the points $\bfx_1(t_1)$ and $\bfx_2(s_{1,i})$.}
\STATE{compute the line $\tilde{\mathcal L}_i$ connecting the points $\bfx_1(t_2)$ and $\bfx_2(s_{2,i})$.}
\STATE{compute the intersection point $\tilde{\bfa}_i$ (possibly at infinity), if any, of ${\mathcal L}_i$ and $\tilde{\mathcal L}_i$; if ${\mathcal L}_i\cap\tilde{\mathcal L}_i=\emptyset$, $i:=i+1$.} 
\STATE{find ${\mathcal P}_{\tilde{\bfa}_i}$.}
\IF{${\mathcal P}_{\tilde{\bfa}_i}(\CCC_1)=\CCC_2$}
\STATE{add ${\mathcal P}_{\tilde{\bfa}_i}$ to the list of projections (initially empty).}
\ENDIF
\ENDFOR
\STATE{{\bf return} the list of projections mapping $\CCC_1$ to $\CCC_2$, or the message {\tt No projection has been found}.}
\end{algorithmic}
\caption*{{\bf Algorithm} 2}
\end{algorithm}

\begin{example}\label{Ex2}
Let $\CCC_1$ be parametrized by $\bfx_1(t)=\left(\frac{-t^2+2t+1}{t^2+1},\frac{t^2+2t-1}{t^2+1},0\right)$, and let $\CCC_2$ be parametrized by $\bfx_2(s)=\left(\frac{1-s^2}{1+s^2},\frac{2s}{1+s^2},1\right)$. One can easily recognize that $\CCC_1$ and $\CCC_2$ are two circles, placed on parallel planes; therefore, we should find two different perspective projections mapping $\CCC_1$ onto $\CCC_2$. In order to check this, we first compute ${\mathcal N}(t,s)$, which yields 
\[{\mathcal N}(t,s)=-4s^2t^2-8s^2t+8st^2+4s^2-16st+4t^2-8s+8t-4.\]Now we proceed with Algorithm 2. In order to do this, we pick $t_1=0$ and $t_2=1$, both satisfying the requirements in Step 2. For $t_1=0$, we get $n_1(s)={\mathcal N}(0,s)=4s^2-8s-4$, which is irreducible over ${\Bbb Q}$. The pair $(\bfx_1(0),\bfx_2(\beta))$, where $4\beta^2-8\beta+4=0$, is lucky; observe that since the polynomial $4\beta^2-8\beta+4$ has two different roots, in fact $(\bfx_1(0),\bfx_2(\beta))$ represents two different pairs, that we can treat at the same time. Also, for $t=t_2$ we get $n_2(s)=-8s^2-16s+8$, which is irreducible over ${\Bbb Q}$ too. The pair $(\bfx_1(1),\bfx_2(\gamma))$, where $-8\gamma^2-16\gamma+8=0$, is also lucky. Furthermore, one can check that $\gamma=-\beta$, so we can write $(\bfx_1(1),\bfx_2(-\beta))$, instead, where $4\beta^2-8\beta+4=0$; as before, $(\bfx_1(1),\bfx_2(-\beta))$ represent two different pairs that we treat simultaneously. Now the parametric representation of the line ${\mathcal L}_{\beta}$ connecting $\bfx_1(0)$ and $\bfx_2(\beta)$ is 
\[{\mathcal L}_{\beta}\equiv \left[1+\mu\cdot \left(1-\frac{1-\beta^2}{1+\beta^2}\right), -1+\mu\cdot\left(-1+\frac{2\beta}{\beta^2+1}\right), -\mu,\right],\]
where $\mu$ is the parameter. The parametric representation of the line $\tilde{\mathcal L}_{\beta}$ connecting $\bfx_1(1)$ and $\bfx_2(-\beta)$ is 
\[\tilde{\mathcal L}_{\beta}\equiv \left[1+\nu\cdot \left(1-\frac{1-\beta^2}{1+\beta^2}\right), 1+\nu\cdot\left(-1+\frac{2\beta}{\beta^2+1}\right), -\nu,\right],\]
where $\nu$ is the parameter. ${\mathcal L}_{\beta}$ and $\tilde{\mathcal L}_{\beta}$ intersect at $(0,0,3+\beta)$. Since the two roots of $4\beta^2-8\beta+4=0$ are real, we have two candidates for the eye point of a projection mapping $\CCC_1$ onto $\CCC_2$, that we can analyze simultaneously. In order to check whether or not $(0,0,3+\beta)$ corresponds to some projection of $\CCC_1$ onto $\CCC_2$, we substitute $\tilde{\bfa}_{\beta}=[0:0:3+\beta:1]$ into Eq. \eqref{matP}, to get 
\begin{equation}\label{matPP}
\hspace{-1.5 cm}P=\begin{bmatrix} 
-(3+\beta)-1 & 0 & 0 & 0\\
0 & -(3+\beta)-1 & 0 & 0\\
0 & 0 & 1 & -(3+\beta)\\
0 & 0 & 1 & -(3+\beta)
\end{bmatrix}.
\end{equation}
Hence, we have 
\[
\begin{array}{l}
{\mathcal P}_{\tilde{\bfa}_{\beta}}(\bfx_1(t))=(\widehat{x}_1(t),\widehat{y}_1(t),\widehat{z}_1(t))=\\
=\displaystyle{\left(\frac{(-2-\beta)(-t^2+2t+1)}{(-3-\beta)(t^2+1)},\frac{(-2-\beta)(t^2+2t-1)}{(-3-\beta)(t^2+1)},1\right)}.
\end{array}
\]
Finally, the gcd of the numerators of the components of $\widehat{\bfx}_1(t)-\bfx_2(s)$ is 
\[-s\beta+\beta t+st-2s+2t+1,\]therefore different from 1. So ${\mathcal P}_{\tilde{\bfa}_{\beta}}(\CCC_1)=\CCC_2$. The whole computation takes 0.234 seconds.
\end{example}

As mentioned before, from a computational point of view the drawback of Algorithm 2 is the fact that it requires computing first a primitive element, and then carry out computations in the new simple algebraic extension. Hence, in general it does not work better than Algorithm 1. However, the idea behind Algorithm 2 is useful to provide an approximate method, which is done in the next subsection.


\subsection{Approximate version of the method.} \label{subsec-approx}

The computation of a simple algebraic extension is necessary when we want to symbolically carry out the intersection of the lines ${\mathcal L}_i$ and $\tilde{\mathcal L}_i$. We can alternatively find an approximation of ${\mathcal L}_i\cap\tilde{\mathcal L}_i$ by proceeding in the following way: (a) approximate the values of $s_i$ and $\tilde{s}_i$, so that we also approximate the equations of the lines ${\mathcal L}_i$ and $\tilde{\mathcal L}_i$; (b) approximate ${\mathcal L}_i\cap\tilde{\mathcal L}_i$, to get approximations of the $\tilde{\bfa}_i$; (c) evaluate ``how close" ${\mathcal P}_{\tilde{\bfa}_i}(\CCC_1)$ and $\CCC_2$ are. 

In order to evaluate the closeness between ${\mathcal P}_{\tilde{\bfa}_i}(\CCC_1)$ and $\CCC_2$, observe first that a projective parametrization of ${\mathcal P}_{\tilde{\bfa}_i}(\CCC_1)$ is provided by $P\cdot \tilde{\bfx}_1(t)$, where $P$ is the matrix in Eq. \eqref{matP}, and 
\[\tilde{\bfx}_1(t)=[x_1(t):y_1(t):z_1(t):1].\]
Therefore, ${\mathcal P}_{\tilde{\bfa}_i}(\CCC_1)$ and $\CCC_2$ are two curves contained in the same plane $\Pi$. To measure how close they are, we can use the \emph{Hausdorff distance} \cite{Alip} between these two curves. More precisely, given a metric space $(X,d)$, the distance between an element $a\in X$ and a non-empty set $B\subset X$ is defined as $d(a,B)=\mbox{inf}_{b\in B}d(a,b)$. Furthermore, given two non-empty subsets $A,B\subset X$, the Hausdorff distance between $A,B$ is 
\[H(A,B)=\mbox{max}\{\mbox{sup}_{a\in A}\{d(a,B)\}, \mbox{sup}_{b\in B}\{d(b,A)\}\}.\]In our case, $d$ will be the usual Euclidean distance, so we will just write $H(A,B)$. The Hausdorff distance is a natural and widely used measure to evaluate the closeness between two objects, and, in particular, to check whether two objects are approximately equal. One can find algorithms to approximate the Hausdorff distance in \cite{Bai}, \cite{Chen}, \cite{Kim}, \cite{J00}, for the case of planar curves; in \cite{Rueda13}, \cite{Rueda}, for space curves; and in \cite{Aspert}, \cite{Barton}, \cite{Brons} for surfaces. Additionally, some theoretical aspects of the question are treated in \cite{Blasco}, for algebraic curves in $n$-space. 

Therefore, in order to decide whether ${\mathcal P}_{\tilde{\bfa}_i}(\CCC_1)\approx \CCC_2$, we compute the Hausdorff distance between ${\mathcal P}_{\tilde{\bfa}_i}(\CCC_1)$ and $\CCC_2$ (by whatever method), and fix a certain tolerance $\epsilon$. Then we will admit ${\mathcal P}_{\tilde{\bfa}_i}(\CCC_1)\approx \CCC_2$ whenever $H({\mathcal P}_{\tilde{\bfa}_i}(\CCC_1), \CCC_2)<\epsilon$. So we get the following algorithm, {\tt Algorithm 3}, to find the ${\tilde{\bfa}_i}$ such that ${\mathcal P}_{\tilde{\bfa}_i}(\CCC_1)\approx \CCC_2$. Notice that this algorithm is also applicable to the case of curves with perturbed coefficients, so that the coefficients of the definining parametrizations are known only up to a certain precision. 

\begin{algorithm}[h!]
\begin{algorithmic}[1]
\REQUIRE Two proper, rational parametrizations $\bfx_1(t),\bfx_2(s)$ defining two space curves $\CCC_1,\CCC_2$, where $\CCC_2$ is planar.
\ENSURE The eye points ${\tilde{\bfa}_i}$, if any, such that ${\mathcal P}_{\tilde{\bfa}_i}(\CCC_1)\approx \CCC_2$, under a certain tolerance $\epsilon$.
\STATE compute the polynomial ${\mathcal N}(t,s)$ as the numerator of the left hand-side of Eq. \eqref{condi}.
\STATE pick $t_1,t_2\in {\Bbb Q}$ such that $t=t_1$, $t=t_2$ are not asymptotes of ${\mathcal N}^{\star}$, and such that no singularity of ${\mathcal N}^{\star}$ lies in $[t_1,t_2]\times {\Bbb R}$.
\STATE Let $s_{j,1},\ldots,s_{j,m}$, $j=1,2$, be the real roots of $n_j(s)={\mathcal N}(t_j,s)$.
\FOR{$i=1,\ldots,m$}
\STATE{check if the pair $(\bfx_1(t_1),\bfx_2(s_{1,i}))$, is lucky; in the negative case, pick a different value for $t_1$.}
\STATE{check if the pair $(\bfx_1(t_2),\bfx_2(s_{2,i}))$, is lucky; in the negative case, pick a different value for $t_2$.}
\STATE{compute the line ${\mathcal L}_i$ connecting the points $\bfx_1(t_1)$ and $\bfx_2(s_{1,i})$.}
\STATE{compute the line $\tilde{\mathcal L}_i$ connecting the points $\bfx_1(t_2)$ and $\bfx_2(s_{2,i})$.}
\STATE{approximate the intersection point $\tilde{\bfa}_i$ of ${\mathcal L}_i$ and $\tilde{\mathcal L}_i$.}
\STATE{compute ${\mathcal P}_{\tilde{\bfa}_i}(\CCC_1).$}
\STATE{compute or estimate $H({\mathcal P}_{\tilde{\bfa}_i}(\CCC_1), \CCC_2)$. }
\IF{$H({\mathcal P}_{\tilde{\bfa}_i}(\CCC_1), \CCC_2)<\epsilon$}
\STATE{add ${\mathcal P}_{\tilde{\bfa}_i}$ to the list of ``approximate" projections (initially empty).}
\ENDIF
\ENDFOR
\STATE{{\bf return} the list of projections such that ${\mathcal P}_{\tilde{\bfa}_i}(\CCC_1)\approx \CCC_2$, or the message {\tt No projection has been found}.}
\end{algorithmic}
\caption*{{\bf Algorithm} 3}
\end{algorithm}

\begin{example}\label{exfin} Let $\CCC_1$ be parametrized by $\bfx_1(t)=\left(\frac{p_1(t)}{p_4(t)},\frac{p_2(t)}{p_4(t)},\frac{p_3(t)}{p_4(t)}\right)$, where
\[
\begin{array}{ccl}
p_1(t) & = & -2t^4-t^3-2t^2-2t,\\
p_2(t) & = & -t^4-t^3-3t^2-2t+1,\\
p_3(t) & = & t^3+2t^2-2t-1,\\
p_4(t) & = & t^4+1.
\end{array}
\]Also, let $\CCC_2$ be parametrized by $\bfx_2(s)=\left(\frac{q_1(s)}{q_4(s)},\frac{q_2(s)}{q_4(s)},\frac{q_3(s)}{q_4(s)}\right)$, where
\[
\begin{array}{ccl}
q_1(s) & = & -s^2(3s^2+2s+3),\\
q_2(s) & = & -(2s^3+6s^2-3),\\
q_3(s) & = & 3s^4+4s^3+9s^2-3,\\
q_4(s) & = & 6s^4+s^3+3s^2+6s+3.
\end{array}
\]
One can check that $\CCC_2$ is the image of $\CCC_1$ under the projection defined by the matrix 
\[
P=\begin{bmatrix} 
-2 & 1 & 1 & 0 \\
1 & -2 & 1 & 0 \\
1 & 1 & -2 & 0 \\
1 & 1 & 1 & -3
\end{bmatrix}.
\]
Furthermore, $\CCC_2$ is contained in the plane $x+y+z=0$. The polynomial ${\mathcal N}(t,s)$ has two irreducible factors over ${\Bbb Q}$, the first one being $t-s$, the second one a dense bivariate polynomial of degree 10. The first factor gives rise to the rational function $\psi(t)=t$. By applying Algorithm 1, one can check that this function corresponds to a perspective projection mapping $\CCC_1$ onto $\CCC_2$ from the affine point $(1,1,1)$. Furthermore, by applying Algorithm 2, one can check that the second factor of ${\mathcal N}(t,s)$ does not provide any other projection. 

Now let us perturb $\CCC_2$ in the following way:
\[
\begin{array}{ccl}
q_1(s) & = & -s^2(3s^2+2s+3)+\frac{1}{1000},\\
q_2(s) & = & -(2s^3+6s^2-3),\\
q_3(s) & = & 3s^4+4s^3+9s^2-3-\frac{1}{1000},\\
q_4(s) & = & 6s^4+s^3+3s^2+6s+3.
\end{array}
\]
For simplicity, we keep the notation $\CCC_2$ for this perturbed curve as well. Now $\CCC_2$ is not \emph{exactly} the projection of $\CCC_1$ anymore; however, it is expectable that ${\mathcal P}_{\tilde{\bfa}_i}(\CCC_1)\approx \CCC_2$ for some $\tilde{\bfa}_i$ close to $(1,1,1)$. In this case, the new polynomial ${\mathcal N}(t,s)$ is absolutely irreducible. We pick $t_1=1$, $t_2=\frac{3}{2}$, which satisfy the conditions in step 2 of Algorithm 3. The real roots of $n_1(s)={\mathcal N}(1,s)$, are 
\[-16.82354557, \mbox{ }-1.536480878, \mbox{ }-0.4983244587, \mbox{ }1.000123802.\]
Also, the real roots of $n_2(s)={\mathcal N}(3/2,s)$, are 
\[-5.921589025, \mbox{ } -1.022763869, \mbox{ } -0.5901004394, \mbox{ } 1.499976466.\]
One can check that the pair $\bfx_1(1),\bfx_2(1.000123802)$, is lucky. Similarly, the pair $\bfx_1(3/2),\bfx_1(1.499976466)$ is also lucky. Furthermore, since the points \[(1,1.000123802),\mbox{ }(3/2,1.499976466)\]are connected by a real branch of ${\mathcal N}^{\star}$, and no singularity of ${\mathcal N}^{\star}$ lies in this branch in between the points, according to Lemma \ref{thesame} the points $\bfx_1(1),\bfx_2(1.000123802)$ and $\bfx_1(3/2),\bfx_1(1.499976466)$ are related by at most one projection, if any. In order to find the eye point of that projection, we compute the lines ${\mathcal L}_1$, $\tilde{\mathcal L}_1$ connecting $\bfx_1(1),\bfx_2(1.000123802)$ and $\bfx_1(3/2),\bfx_1(1.499976466)$, respectively. One can check that these lines are skew, but the least squares method provides the solution 
\[\tilde{a}_1=(0.998393660, 0.998630969, 0.9998988141)\]for the intersection point of ${\mathcal L}_1$, $\tilde{\mathcal L}_1$. Notice that this point is really very close to the point $(1,1,1)$. The projection ${\mathcal P}_{\tilde{\bfa}_1}(\CCC_1)$ is the curve $\tilde{\CCC}_2$, parametrized by $\tilde{\bfx}_2(s)=\left(\frac{\tilde{q}_1(s)}{\tilde{q}_4(s)},\frac{\tilde{q}_2(s)}{\tilde{q}_4(s)},\frac{\tilde{q}_3(s)}{\tilde{q}_4(s)}\right)$, where
\[
\begin{array}{ccl}
\tilde{q}_1(s) & = & 2.998665906s^4+1.998529783s^3+2.998665906s^2+0.003484926s,\\
\tilde{q}_2(s) & = & 0.001030536s^4+1.998292474s^3+5.994877422s^2+0.002061072s\\
& & -2.996923443,\\
\tilde{q}_3(s) & = & -2.999696442s^4-3.996822257s^3-8.993543328s^2-0.005545998s\\
& & +2.996923443,\\
\tilde{q}_4(s) & = & -5.996923443s^4-s^3-3s^2-6s-2.996923443.
\end{array}
\]
Although $\tilde{\CCC}_2\neq \CCC_2$, one can check that the infinity norm of $q_i(t)-\tilde{q}_i(t)$ is small, suggesting that ${\mathcal P}_{\tilde{\bfa}_1}(\CCC_1)\approx \CCC_2$. More precisely, in Figure \ref{ex3} we have plotted $\CCC_2$ (in blue color) and $\tilde{\CCC_2}$ (in red color). Furthermore, we have plotted $\tilde{\CCC}_2$ over a smaller parameter interval than in the case of $\CCC_2$, so that we can appreciate the almost perfect overlapping between $\CCC_2$ and $\tilde{\CCC}_2$. Furthermore, using the technique of \cite{Rueda13}, the Hausdorff distance between $\CCC_2$ and $\tilde{\CCC}_2$ is $9.72\cdot 10^{-6}$. Therefore, if we choose $\epsilon=0.001$, which is the size of the perturbation we introduced in the second curve, we have that ${\mathcal P}_{\tilde{\bfa}_2}(\CCC_1)\approx \CCC_2$.

\begin{figure}
$$\begin{array}{c}
\includegraphics[scale=0.3]{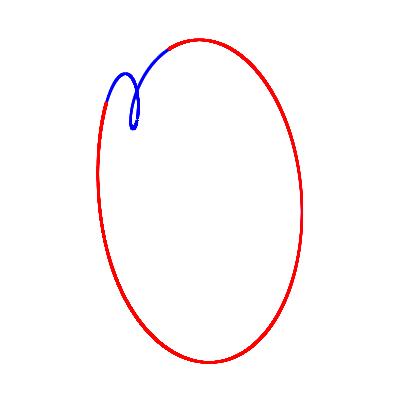}
\end{array}$$
\caption{The curves $\CCC_2$ and $\tilde{\CCC}_2$}\label{ex3}
\end{figure}

Similarly, one can check that the pairs $\bfx_1(1)$, $\bfx_2(-0.4983244587)$ and $\bfx_1(3/2)$, $\bfx_1(-0.5901004394)$ are also lucky. By proceeding as before, we get the tentative eye point 
\[\tilde{\bfa}_2=(13.96606371, 26.47971848, -6.649856672).\]The projection ${\mathcal P}_{\tilde{\bfa}_2}(\CCC_1)$ is the curve $\widehat{\CCC}_2$, parametrized by $\widehat{\bfx}_2(s)=\left(\frac{\widehat{q}_1(s)}{\widehat{q}_4(s)},\frac{\widehat{q}_2(s)}{\widehat{q}_4(s)},\frac{\widehat{q}_3(s)}{\widehat{q}_4(s)}\right)$, where
\[
\begin{array}{ccl}
\widehat{q}_1(s) & = & 25.69365991s^4+19.82986181s^3+25.69365991s^2-16.20453122s,\\
\widehat{q}_2(s) & = & -45.64322992s^4+7.316207038s^3+21.94862111s^2-91.28645982s\\
& & -33.79592552,\\
\widehat{q}_3(s) & = & 19.94957001s^4-27.14606885s^3-47.64228102s^2+107.4909911s\\
& & +33.79592552,\\
\widehat{q}_4(s) & = & -36.79592552s^4-s^3-3s^2-6s-33.79592552.
\end{array}
\]
In this case, the infinity norm of $q_i(t)-\widehat{q}_i(t)$ is not small. Furthermore, in Figure \ref{ex3-2} we have plotted $\CCC_2$ (in blue color) and $\widehat{\CCC}_2$ (in red color): one can see that both curves are very different. Furthermore, using the technique of \cite{Rueda13}, the Hausdorff distance between $\CCC_2$ and $\widehat{\CCC}_2$ is $0.253>0.001$. Therefore, ${\mathcal P}_{\tilde{\bfa}_2}(\CCC_1)\not\approx \CCC_2$. The other tentative eye points behave in a similar way.

\begin{figure}
$$\begin{array}{c}
\includegraphics[scale=0.3]{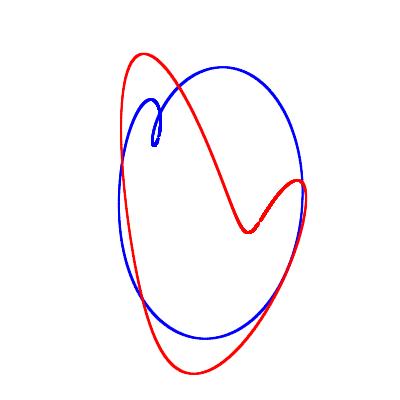}
\end{array}$$
\caption{The curves $\CCC_2$ and $\widehat{\CCC}_2$}\label{ex3-2}
\end{figure}
  
\end{example}

\section{Conclusions and Further Work.}\label{conclusion}

We have presented three algorithms, two of them symbolic and one of them approximate, to check whether or not a planar, rational curve $\CCC_2$ properly parametrized, is the (parallel or perspective) projection of another rational curve $\CCC_1$, also properly parametrized, non-necessarily planar. In the affirmative case, the algorithms compute all the projections mapping $\CCC_1$ onto $\CCC_2$. All the algorithms are based on the fact that behind each projection, we have a rational mapping between the parameter spaces of $\CCC_1$ and $\CCC_2$. The first algorithm, {\tt Algorithm 1}, uses bivariate factoring. In order to compute the projections from eye points with rational coordinates, standard bivariate factoring over the rationals suffices. However, in the most general case we need to compute an absolute factorization of a bivariate polynomial. The second algorithm, {\tt Algorithm 2}, does not use factoring, but in turn it requires, in general, to compute the primitive element of a double field extension. Except for easy, low degree cases, this operation can be costly. However, this second algorithm gives rise to a third, approximate, algorithm, {\tt Algorithm 3}. This last algorithm avoids the computation of the primitive element by changing to a floating point setting. In this last case, we compute approximations for tentative eye points, and then we evaluate whether $\CCC_2$ is ``aproximately" the projection of $\CCC_1$. In order to do this, we compute the Hausdorff distance between $\CCC_2$ and the projection of $\CCC_1$ from the tentative eye point.  

The algorithms strongly exploit the fact that $\CCC_1$ and $\CCC_2$ are rational curves. Therefore, as such they cannot be generalized to the case when $\CCC_1$ and $\CCC_2$ are implicit, not necessarily rational, algebraic curves. Hence, an interesting problem to be attacked in the future is the development of algorithms for solving the problem in the implicit case.

\end{document}